\documentclass[11pt,A4paper]{article}
\usepackage[left=32mm,right=32mm,top=30mm,bottom=32mm]{geometry}
\usepackage{psfrag}
\usepackage{epigraph}
\usepackage[english]{babel}
\usepackage{epsfig}
\usepackage{amsfonts} 
\usepackage{mathrsfs} 
\usepackage{graphicx, color, subfigure, booktabs, xypic}
\usepackage{amssymb, array}  
\usepackage{stmaryrd}
\usepackage{MnSymbol}
\usepackage{amsmath,amscd, amsthm, dsfont} 
\usepackage{color}
\numberwithin{equation}{section}
\newcommand{\p}{\ensuremath{\mathbb{P}}}
\newcommand{\s}{\ensuremath{\mathscr{S}}}

\newcommand{\Mod}{\ensuremath{\mathrm{Mod}}}

\newtheorem{theorem}{Theorem}[section]
\newtheorem{lemma}[theorem]{Lemma}
\newtheorem{proposition}[theorem]{Proposition}
\newtheorem{definition}[theorem]{Definition}
\newtheorem{corollary}[theorem]{Corollary}

\newtheorem{remark}[theorem]{Remark}

\newtheorem*{thma}{Theorem 1}
\newtheorem*{thmb}{Theorem 2}

\theoremstyle{definition}

\title{Combinatorial rigidity of arc complexes}

\author{Valentina Disarlo} 
\date{}
\begin{document}

\maketitle

\begin{abstract}
We study the arc complex of a surface with marked points in the interior and on the boundary. We prove that the isomorphism type of the arc complex determines the topology of the underlying surface, and that in all but a few cases every automorphism is induced by a homeomorphism of the surface. As an application we deduce some rigidity results for the Fomin-Shapiro-Thurston cluster algebra associated to such a surface. Our proofs do not employ any known simplicial rigidity result. 
\end{abstract}

   
\section{Introduction}
The \emph{arc complex} of a surface with marked points is a simplicial complex whose vertices are the homotopy classes of essential arcs based on the marked points, and whose simplices correspond to families of pairwise disjoint arcs on the surface. Its dual is the \emph{flip graph}, a graph that encodes the combinatorics of ideal triangulations. Both objects were introduced by Harer \cite{Har2} as a tool to study some homological properties of the mapping class group. Their combinatorial and topological properties are crucial in Penner's decorated Teichm\"uller theory \cite{Pen3}, in the construction of various compactifications of the moduli space \cite{BoeP}, and in Fomin-Shapiro-Thurston's theory of cluster algebras associated to bordered surfaces \cite{DTh, DTh2}. Some geometric properties of these graphs and their relation with the geometry of the mapping class group were recently investigated by Hensel-Przytycki-Webb \cite{Hensel} and Disarlo-Parlier \cite{Disarlo-Parlier}.   \\ 

In this paper we deal with the arc complex of an orientable surface with non-empty boundary, with at least one marked point on each boundary component and a finite number of punctures in its interior. These surfaces are often called \emph{ciliated} in literature (see for instance Fock-Goncharov \cite{Fock1}), and we will refer to them as pairs $(S,\p)$, $S$ being a punctured surface with boundary and $\p$ a set of marked points on the boundary of $S$. In this paper, we will be primarily interested in the combinatorial rigidity of their arc complexes. The problems we will deal with are motivated by the theory of cluster algebras. Similar rigidity problems were studied by many authors in the past and yield various descriptions of the mapping class group as the automorphism group of some simplicial complex naturally associated to a surface (for a survey of these results see \cite{MCP}). Most of the proofs rely on non-trivial reductions to the rigidity theorem for curve complexes. To our knowledge all these results refer to the mapping class of closed or punctured surfaces, and we are not aware of similar results for the mapping class group of ciliated surfaces. Our proofs do not employ any previously known rigidity result. \\ 

The first question we approach is the following: can two non-homeomorphic ciliated surfaces have isomorphic arc complexes? The answer is ``no". 
\begin{theorem}
Let $(S,\p)$ and $(S',\p')$ be two ciliated surfaces. If their arc complexes $A(S,\p)$ and $A(S',\p')$ are isomorphic, then $(S,\p)$ and $(S',\p')$ are homeomorphic.  
\end{theorem}
A similar result also holds for the curve complex and the proof relies on the fact that isomorphic complexes have the same dimension. Here this fact alone is not sufficient because there exists many pairs of non-homeomorphic surfaces whose arc complexes have the same dimension, and more machinery is required in our proof. \\ 

The second question we approach is the classical Ivanov's problem for the arc complex of a ciliated surface: is every automorphism of such a complex induced by a homeomorphism of its underlying surface? The answer is ``yes". In fact, in all but a finitely many cases automorphism group of the arc complex is isomorphic to the mapping class group of the surface, except a few cases. We also list and study these exceptional cases. 
\begin{theorem}
Every automorphism of the arc complex $A(S,\p)$ is induced by a mapping class of $(S,\p)$. Moreover, if the dimension of $A(S,\p)$ is greater than 1, the automorphism group of $A(S,\p)$ is isomorphic to the mapping class group of $(S,\p)$. 
\end{theorem}
The mapping class group we consider here is the group of homeomorphisms of the surface that preserve setwise the marked points and the punctures, modulo homotopies fixing 
the marked points and the punctures pointwise. Mapping classes are allowed to permute the marked points, the punctures and to reverse the orientation of the surface. The Dehn 
twists around the boundary components of the surfaces are here nontrivial elements.In \cite{IMC} Irmak-McCarthy proved an analogous theorem for the arc complex of a punctured 
surface without boundary, but their results do not imply ours. 
We remark that theorem 1.1 and theorem 1.2 together imply that every isomorphism between two arc complexes arises from a homeomorphism between the underlying surfaces. This 
is related to a question of Aramayona-Souto \cite{Javi} concerning the superrigidity of some simplicial complexes associated to a surface. By a construction of Korkmaz-
Papadopoulos \cite{KP2},  Theorem 1.2 implies immediately that every automorphism of the flip graph $F(S,\p)$ is induced by a mapping class of $(S,\p)$. In \cite{Hugo-3} 
Aramayona-Koberda-Parlier also employ our Theorems 1.1 and 1.2 above to characterize all injective simplicial maps between different flip graphs. Among other things, they 
prove that these maps only arise from topological embeddings between the underlying two surfaces. \\ 

As above mentioned, our results have further applications in the context of the Fomin-Shapiro-Thurston cluster algebra associated to a ciliated surface. Indeed, the cluster complex of this cluster algebra is a slightly modified arc complex which has a natural projection to our arc complex. Similarly, the exchange graph of this cluster algebra is a slightly modified flip graph with a natural projection  to our flip graph. When the surface has no punctures, both projections are isomorphisms (see  \cite{DTh}).
In the light of this, Theorem 1.1 proves that the combinatorics of the cluster complex alone suffices to determine the topology of the underlying surface. We have: 
\begin{corollary}
Let $(S,\p)$ and $(S',\p')$ be two ciliated surfaces without punctures, and let  $\mathscr A(S,\p)$ and $\mathscr A(S',\p')$ be their cluster algebras. If the cluster complexes of $\mathscr A(S,\p)$ and $\mathscr A(S',\p')$ are isomorphic, then $(S,\p)$ and $(S',\p')$ are homeomorphic.
\end{corollary}
As an immediate corollary we have a rigidity result for these cluster algebras:
\begin{corollary}
Let $(S,\p)$ and $(S',\p')$ be two ciliated surfaces without punctures. If their cluster algebras $\mathscr A(S,\p)$ and $\mathscr A(S',\p')$  are strongly isomorphic, then $(S,\p)$ and $(S',\p')$ are homeomorphic.
\end{corollary}

Theorem 1.2 also implies the following corollaries that further prove the ``naturality" of Fomin-Shapiro-Thurston's construction. 
\begin{corollary} 
If $(S,\p)$ is a ciliated surface without  punctures then the automorphism group of the cluster complex of $\mathscr A(S,\p)$ is isomorphic to the mapping class group of $(S,\p)$. 
\end{corollary}

\begin{corollary} 
If $(S,\p)$ is a ciliated surface without  punctures then the automorphism group of the exchange graph of $\mathscr A(S,\p)$ is isomorphic to the mapping class group of $(S,\p)$. 
\end{corollary}

The structure of this paper is the following. In Section \ref{ch1-section2} we introduce the notation, we list exceptional cases, we recall the Fomin-Shapiro-Thurston construction and we prove the corollaries above. In Section \ref{ch1-section-1.2.3} we discuss some invariance lemmas that will be used throughout the paper and we prove Theorem 1.1. In Section \ref{section3} we prove Theorem 1.2. 

\paragraph{Acknowledgements} This paper evolves from a part of my PhD. thesis. I would like to thank Athanase Papadopoulos for introducing me to this field, Mustafa Korkmaz for his encouragement and comments on an earlier version of this paper, Dylan P. Thurston for the enlightening conversations on his works with Fomin and Shapiro, and an anonymous referee for the useful comments on an earlier version of this paper.
\section{Combinatorics of arc complexes}\label{ch1-section2}

\subsection{Topological set-up} 
In this paper we will use the following notation (see also Fomin-Shapiro-Thurston \cite{DTh}). 
Let $S_{g,b}^s$  ($S$ for short) be a compact orientable surface of genus $g\geq 0$, whose boundary $\partial S$ has $b > 0$ boundary components $\mathscr B_1, \ldots , \mathscr B_b$. 
We assume that $S$  has $s\geq 0$ marked points in its interior and each $\mathscr B_i$ has exactly $p_i\geq 1$ marked points. We will denote by $\p= (p_1, \ldots , p_b) \in \mathbb N^b  \setminus \{0\}$ the vector of the marked points. 
We will denote by $\mathscr{P}$ the set of marked points on $\partial S$ and by $\mathscr S$ the set of marked points in the interior of $S$. We will often refer to the points in $\mathscr S$ as \emph{punctures}. A \emph{bordered surface with punctures and marked points} (or a \emph{ciliated surface}) is such a pair $(S_{g,b}^s, \p)$. We will often refer to the disk $(S_{0,1}^s, (n))$ as the $n$-\emph{polygon} with $s$ punctures. \\ 

In this paper we will be interested in the combinatorics of the arcs on $(S,\p)$. Every arc here is \emph{essential}, that is, it enjoys the following properties:   
\begin{itemize}
\item its endpoints  are marked points in $\mathscr P \cup \mathscr S$;
\item it does not intersect $\mathscr P \cup \mathscr S$, except at its endpoints; 
\item it does not self-intersect, possibly except at its endpoints;
\item it does not intersect the boundary of $S$, possibly except at its endpoints. 
\item it is not homotopic to a segment of boundary between two consecutive points of $\mathscr P$ on the same boundary component of $S$. 
\end{itemize}
All the arcs here will be considered up to isotopies among arcs of the same type fixing $\mathscr P \cup \mathscr S$ pointwise. If the endpoints of an arc coincide, we will often call it a \emph{loop}. 

\begin{definition}
Let $\alpha, \beta$ be two arcs in $(S, \p)$. We define their \emph{intersection number} $\iota(\alpha,\beta)$ as follows: 
$$\iota(\alpha,\beta) = \mathrm{min } | \mathring a \cap \mathring b | ,$$ 
where $a$ is an essential arc in the isotopy class of $\alpha$ and $\mathring a$ is its interior, and $\beta$ is an essential arc in the isotopy class of $\beta$ ($\mathring b$ is its interior). 
\end{definition}

We will say that $\alpha$ and $\beta$ are \emph{disjoint} when $\iota(\alpha, \beta) =0$, that is, $\alpha$ and $\beta$ have no intersection in the interior. Disjoint arcs are allowed to share one or both endpoints. A maximal collection of pairwise disjoint essential arcs is called a \emph{triangulation} of $(S,\p)$. The arcs of a triangulation cut $(S,\p)$ into \emph{triangles}. Triangles can be embedded or immersed. Self-folded triangles and triangles with one or two sides on the boundary are also allowed here. By the Euler characteristic formula each triangulation of $(S,\p)$ contains $ 6g +3b + 3s +  |\mathscr P| - 6$ arcs. Two triangulations of $(S,\p)$ are obtained from each other by a \emph{flip} when they are the same except for a quadrilateral where one diagonal is replaced by the other one. The following are well-known facts (see Farb-Margalit \cite{FM}) we will use later.
\begin{proposition}\label{connected}
The following holds:
\begin{enumerate}
\item Any set of pairwise disjoint arcs in $(S,\p)$ can be extended to a triangulation. 
\item If two triangulations of $(S,\p)$ share all arcs but one, then the two arcs intersect exactly once and the two triangulations differ by a flip. 
\item Any two triangulations on a surface differ by a finite number of flips.\end{enumerate}
\end{proposition}

\subsection{The arc complex $A(S,\p)$}\label{ch1-1}
In this section we will define the arc complex, recall some known properties and list some examples in low dimensions. 

\begin{definition}
The arc complex $A(S, \p)$ is the simplicial complex whose vertices are the equivalence classes of arcs with endpoints in $\mathscr P \cup \s$ modulo isotopy fixing $\mathscr{P}\cup \s$ pointwise, and whose simplices are defined as follows.
A set of vertices $\langle  a_1, \ldots, a_k  \rangle$  spans a $(k-1)$-simplex if and only if $a_1, \ldots, a_k$ can be realized simultaneously as arcs on $(S,\p)$ mutually disjoint in the interior. 
\end{definition}

A simplex of maximal dimension in $A(S,\p)$ corresponds to a triangulation of $(S, \p)$. 
Proposition \ref{connected} and the Euler characteristic formula imply the following.
\begin{proposition}
The following holds: 
\begin{enumerate}
\item every simplex in $A(S,\p)$ can be extended to a simplex of maximal dimension.
\item each simplex of codimension 1 is a face of at most two maximal simplices. 
\item the one skeleton of the cellular complex dual to $A(S,\p)$ is arcwise connected.
\end{enumerate}
\end{proposition}
\begin{lemma}\label{dimension}The following holds: 
\begin{enumerate}
\item The arc complex $A(S,\p)$ has finitely many vertices if and only if $(S, \p)$ is a polygon with at most one puncture in the interior. 
\item The dimension of $A(S,\p)$ is equal to  $6g + 3b + 3s + |\mathscr{P}| - 7$.
\end{enumerate}
\end{lemma}
The arc complex can be endowed with the natural shortest path distance where every edge has length one. The following is a well-known result (see for instance \cite{Hensel}). 
\begin{lemma}\label{infinite diameter}
If every boundary component of $(S,\p)$ has exactly one marked point (i.e.  $\p = (1, \ldots, 1)$) and $S$ is not a once-punctured disk, then $A(S,\p)$ has infinite diameter. 
\end{lemma}


The topology of the arc complex was studied by Hatcher \cite{Hat}, who proved the following. 
\begin{proposition}[Hatcher \cite{Hat}]
If $A(S, \p)$ has dimension at least 1, then it is arcwise connected. Except when $S$ is a disk or an annulus with $s=0$, $A(S, \p)$ is contractible. If $S$ is a disk with $s=0$, $A(S, \p)$ is PL-homeomorphic to a sphere. 
\end{proposition}

We will also be interested in the flip graph. 
\begin{definition}
Assume that $\dim A(S,\p) \geq 1$.The flip graph $F(S,\p)$ is the graph whose vertices correspond to ideal triangulations of $(S,\p)$ and whose edges correspond to flips between triangulations. In other words, $F(S,\p)$ is the 1-skeleton of the CW complex dual to $A(S,\p)$. 
\end{definition}

By Lemma \ref{connected} $F(S,\p)$ is arcwise connected. This graph enjoys many nice geometric properties (see Disarlo-Parlier \cite{Disarlo-Parlier} , Aramayona-Koberda-Parlier \cite{Hugo-3}). Korkmaz-Papadopoulos \cite{KP2} studied the automorphism group of the flip graph of a punctured surface. They did not consider ciliated surfaces as we do, nevertheless their proof applies word-by-word to our setting. We have:

\begin{theorem}[Korkmaz-Papadopoulos \cite{KP2} ]\label{KP}
If $\dim A(S,\p) \geq 1$, the automorphism group of $F(S,\p)$ is isomorphic to the automorphism group of $A(S,\p)$. 
\end{theorem}

\subsubsection{Low dimensional examples} Using the dimension formula for $A(S,\p)$ and studying the cases, we have the following.
\begin{lemma}[Low dimensional cases]\label{remark1} The following holds (see Figure \ref{low-dim-1}):
\begin{enumerate}
\item $A(S,\p) = \varnothing$ if and only if $(g,b,s) = (0,1,0)$ and $p_1 \in \{1,2,3 \}$. 
\item  $A(S,\p)$ has $\dim A(S,\p) = 0$ if and only if  $(S, \p)$ is one of the following: 
\begin{itemize}
\item $(g,b,s,\p) = (0,1,1;(1))$ (i.e. it is a 1-polygon with one puncture), and in this case $A(S_{0,1}^{1};(1))$ is a single vertex. 
\item $(g,b,s,\p)  = (0,1,0;(4))$ (i.e. it is a 4-polygon), and in this case $A(S_{0,1}^{0},(4))$ consists of two disjoint vertices. 
\end{itemize}
\item $A(S,\p)$ has $\dim A(S,\p) = 1$ if and only if $(S,\p)$ is one of the following: 
\begin{itemize}
\item $(g,b,s , \p ) = (0,1,0;(5))$ (i.e. it is a 5-polygon), in this case $A(S_{0,1}^{0};(5))$  is isomorphic to a segment of diameter 2.
\item  $(g,b,s, \p) =  (0,2,0;(2))$ (i.e. it is an annulus with one marked point on each boundary component), in this case $A(S_{0,2}^{0}, (1,1))$ is isomorphic to $\mathbb R$;
\item $(g,b,s, \p) = (0,1,1;(2))$ (i.e. it is a 2-polygon with one puncture), in this case  $A(S_{0,1}^{1};(2))$ is isomorphic to a circle and has diameter 3.
\end{itemize}
\end{enumerate}
\end{lemma}

\begin{figure}[htbp]
\begin{center}
\psfrag{S}{\small $(S,\p)$}
\psfrag{A}{\small $A(S,\p)$}
\includegraphics[width=8 cm]{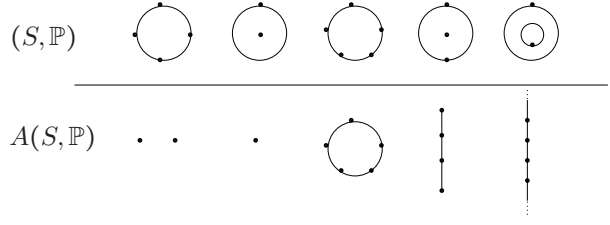}
\caption{Arc complexes of dimension at most $1$}\label{low-dim-1}
\end{center}
\end{figure}




\subsubsection{The mapping class group $\mathrm{Mod}(S, \p)$ action} 

Let $\mathrm{Homeo}(S,\p)$ be the group of homeomorphisms of $S$ fixing $\mathscr{P}\cup \mathscr S$ as a set. Let $\mathrm{Homeo}_0(S,\p) \subseteq \mathrm{Homeo}(S,\p)$ be the normal subgroup consisting of homeomorphisms isotopic to the identity through isotopy fixing $\mathscr P \cup \mathscr S$. 
The \emph{mapping class group of } $(S,\p)$ is the quotient group $\mathrm{\mathrm{Mod}}(S, \p)= \mathrm{Homeo}(S, \p)/\mathrm{Homeo}_0(S,\p)$. 
The \emph{pure mapping class group of }$(S,\p) $ is its subgroup $\mathrm{PMod}(S,\p)$ generated by the homeomorphisms  fixing the boundary of $S$ and the punctures $\mathscr{S}$ pointwise. The \emph{orientation preserving mapping class group} $\mathrm{Mod}^+(S,\p)$ is the subgroup of $\mathrm{Mod}(S,\p)$ generated by orientation-preserving homeomorphisms. The subgroup $\mathrm{Mod}^+(S,\p)$ has index 2 in $\mathrm{Mod}(S,\p)$. 

Denote by $\mathfrak{S} _{n}$ be the symmetric group on $n$ elements. For every $i=1, \ldots, b$, let $r_i$ be the number of boundary components having exactly $p_i$ marked points. The following proposition is immediate. 
\begin{proposition}\label{ES}
There is a short exact sequence: 
$$ 1 \to \mathrm{PMod}(S, \p)\to \mathrm{\mathrm{Mod}}^+(S, \p) \to \bigoplus_{i=1}^b (\mathfrak{S} _{r_i} \ltimes \mathbb Z_{p_i})\oplus \mathfrak{S} _{s} \to 1 .$$ 
\end{proposition}
We remark that $\mathrm{Mod}(S, \p)$ and its subgroups act naturally on the set of isotopy classes of arcs and triangulations. This action extends to a simplicial action on $A(S,\p)$ and $F(S,\p)$. There is a natural group homomorphism $ \rho: \mathrm{Mod}(S, \p) \to \mathrm{Aut} A(S,\p)$, where $\mathrm{Aut} A(S,\p)$ is the automorphism group of $A(S,\p)$.
Similarly, there is a natural group homomorphism $ \rho': \mathrm{Mod}(S, \p) \to \mathrm{Aut} F(S,\p)$, where $\mathrm{Aut} F(S,\p)$ is the automorphism group of $F(S,\p)$.

\begin{proposition}\label{one-to-one}
If $\dim A(S,\p) \geq 2$ then $\rho$ and $\rho'$ are injective.
\end{proposition}

\begin{proof}
A mapping class in the kernel of $\rho$ is a mapping class that fixes the isotopy class of every arc on the surface, in particular it fixes the isotopy type of every triangulation. If $\dim A(S,\p) \geq 2$ then one can use Alexander's lemma to conclude that such a mapping class is the identity. The same holds for $\rho'$. 
\end{proof}

\subsection{The cluster algebra of a ciliated surface}
In \cite{DTh, DTh2} Fomin-Shapiro-Thurston define and study the properties of a cluster algebra $\mathscr A(S_{g,b}^s, \p)$ that is naturally associated to a ciliated surface $(S_{g,b}^s, \p)$ as in our setting. In particular, they proved that this cluster algebra provides positive examples for many conjectures in cluster algebra theory. Let us quickly recall what a cluster algebra is (for a more detailed introduction see Fomin-Zelevinsky \cite{Fomin-0}). A cluster algebra is a commutative ring equipped with a distinguished set of generators (\emph{cluster variables}) grouped into overlapping subsets (\emph{clusters}) of the same cardinality (\emph{rank}). A \emph{seed} for a cluster algebra of rank $n$ over $\mathbb Q$ is a pair $(\tau, B)$ where $\tau$ is a cluster made of $n$ algebraically independent 
rational functions and $B$ is a skew-symmetric $n \times n$ matrix. A \emph{mutation} on $(\tau , B)$ is an operation that transforms $(\tau, B$) into a new seed $(\tau', B')$ using 
very nice formulae encoded in the matrix $B$. The cluster algebra $\mathscr A(\tau, B)$ is the sub algebra of the field of rational functions over $\mathbb Q $ generated by all the clusters in the seeds that differ from $(\tau, B)$ by a finite number of mutations. The \emph{exchange graph} of a cluster algebra is a graph whose vertices correspond to seeds and whose edges correspond to mutations. The \emph{cluster complex} of a cluster algebra is a complex whose vertices correspond to cluster variables and whose maximal simplices correspond to seeds. 
Two cluster algebras are \emph{(strongly) isomorphic} if they are isomorphic as algebras under an isomorphism that maps clusters to clusters.  

Let us now quickly describe the Fomin-Shapiro-Thurston construction in the case when the surface has no puncture in the interior (for details see \cite{DTh}). Fix a triangulation $\tau$ of $(S, \p)$ and label its arcs by the numbers $1, \ldots , N=6g+3b+3s+|\mathscr P| -6$. For each ideal triangle $\Delta$ in $S \setminus \tau$, define the $N \times N$ matrix $B^\Delta = (b_{ij}^\Delta)$ as follows 
$$b_{ij}^\Delta = \begin{cases}
1  & \mbox{ if $\Delta$ has sides labelled $i$ and $j$ with $j$ following $i$ in the clockwise order} \\ 
-1 & \mbox{ if the same holds with counterclockwise order} \\ 
0 & \mbox{ otherwise }
\end{cases} $$

The matrix defined as  $B(\tau) = \sum_{\Delta} B^{\Delta} $ is skew-symmetric. Choose a variable $t_i$ for each arc of the triangulation, and denote by $\tau$ the set of all the variables obtained. The pair $(\tau, B(\tau))$ is a seed and it defines a cluster algebra $\mathscr A_\tau(S,\p) = A(\tau, B(\tau))$ over $\mathbb Q$ of rank $N=6g+3b+3s+|\mathscr P| -6$. 
 
\begin{theorem}[Fomin-Shapiro-Thurston \cite{DTh} ]\label{Dylan} The cluster algebra $\mathscr A_\tau(S,\p)$ does not depend on the choice of $\tau$. Moreover, we have the following: 
\begin{itemize}
\item the cluster complex of $\mathscr A (S,\p)$ is isomorphic to the arc complex $A(S,\p)$; 
\item the exchange graph of $\mathscr A (S,\p)$ is isomorphic to the flip graph $F(S,\p)$. 
\end{itemize}
Under this isomorphism: 
\begin{itemize} 
\item cluster variables correspond to arcs; 
\item cluster seeds correspond to triangulations; 
\item mutations correspond to flips. 
\end{itemize}
\end{theorem}   

Put in this setting, our rigidity theorems imply the following (Corollaries D -- G): 
\begin{corollary}\label{Dylan2}
If $(S,\p)$ and $(S',\p')$ are ciliated surfaces without punctures then the cluster complexes of their cluster algebras $\mathscr A(S,\p)$ and $\mathscr A(S',\p')$ are strongly isomorphic if and only if $(S,\p)$ and $(S',\p')$ are homeomorphic. 
\end{corollary}

\begin{proof}
By Theorem \ref{Dylan} every isomorphism between the two cluster complexes induces an isomorphism between the two arc complexes  $A(S,\p)$ and $A(S',\p')$. By Theorem 1.1, this implies that $(S,\p)$ and $(S',\p')$ are homeomorphic. 
\end{proof}

\begin{corollary}
If $(S,\p)$ and $(S',\p')$ are ciliated surfaces without punctures then their cluster algebras $\mathscr A(S,\p)$ and $\mathscr A(S',\p')$ are strongly isomorphic if and only if 
$(S,\p)$ and $(S',\p')$ are homeomorphic. 
\end{corollary}

\begin{proof}
By Theorem 7.11 and 9.21 in \cite{DTh}, a strong isomorphism between $\mathscr A(S,\p)$ and $\mathscr A(S',\p')$ induces an isomorphism between their two cluster complexes. The rest follows by corollary \ref{Dylan2}
\end{proof}

The following result is a straightforward application of Theorems \ref{Dylan}, Theorem 1.2 and Theorem \ref{KP}.  
\begin{corollary}
If $(S,\p)$ is a cilated surface without punctures then every automorphism of the cluster complex of $A(S,\p)$ is induced by an element of $\Mod(S, \p)$. Similarly, every automorphism of the exchange graph of $A(S,\p)$ is induced by an element of $\Mod(S, \p)$. 
\end{corollary}

Fomin-Shapiro-Thurston also provide a similar construction for ciliated surfaces with punctures.  In this case the cluster complex is the \emph{tagged arc complex} $A^\Join(S, \p)$, a slightly modified arc complex whose vertices are triples $(\alpha, \tau_1, \tau_2)$ where $\alpha$ is a vertex in $A(S,\p)$ and $\tau_1, \tau_2 \in \{ \mbox{plain, notched}\} 
$ are tags at its endpoints. Similarly, its exchange graph is the \emph{tagged flipped graph} $F^\Join(S, \p)$, the 1-skeleton of the cellular complex dual to $A^\Join(S,\p)$. The compatibility relations that define simplices in this case is more technical and complicated (for details see \cite{DTh}), but there are natural 
projection maps $A^\Join(S, \p) \to A(S, \p)$ and $F^\Join(S, \p) \to F(S,\p)$.

\section{Proof of Theorem 1.1}\label{ch1-section-1.2.3}
In Subsection \ref{thmA:subsec1} we recall the definitions of links and join products of simplicial complexes, in Subsection \ref{sec:inv_lemma} we prove some invariance lemmas and Theorem 1.1.

\subsection{Links of vertices}\label{thmA:subsec1}
Let us first recall some standard definitions in simplicial topology.
\begin{definition} 
Let $K$ be a nonempty simplicial complex and let $\sigma$ be one of its simplices. The \emph{link} $\mathrm{Lk}(\sigma, K)$ of $\sigma$ (or $Lk(\sigma)$ for short) is the subcomplex of $K$ whose simplices are the simplices $\tau$ such that $\sigma \cap \tau = \varnothing$ and $ \sigma \cup \tau$ is a simplex of $K$. 
\end{definition}

\begin{definition}
Let $K$ and $H$ be two simplicial complexes whose vertex sets $V_1$ and $V_2$ are disjoint. 
The \emph{join of $K$ and $H$} is the simplicial complex $K \star H$ with vertex set $V_1 \cup V_2$; a subset $\sigma$ of $V_1 \cup V_2$ is a simplex of $K \star H$ if and only if $\sigma$ is a simplex of $K$, a simplex of $H$ or the union of a simplex of $K$ and a simplex of $H$. We have $\dim(K \star H) = \dim K + \dim H + 1$. 
A \emph{cone} $C(K)$ over $K$ a is simplicial complex isomorphic to $K \star \{ w\}$, for some vertex $w$.  
\end{definition}

The following remark will be very useful later. 
\begin{remark}
If $Z$ is a simplicial complex isomorphic to the join product of two non-empty simplicial complex, then $Z$ is arcwise connected and has diameter 2 with respect to its shortest path distance. 
\end{remark}

The link of a simplex in $A(S,\p)$ isomorphic to (the join of) the arc complex(es) of the ciliated surface(s) obtained by cutting along the arcs that span the simplex and recording the marked points involved in the cut (see  
 Figure \ref{arco-1}).
 
\begin{figure}[htp!]
\begin{center}
\psfrag{v}{ \tiny $v$}
\includegraphics[width=16cm]{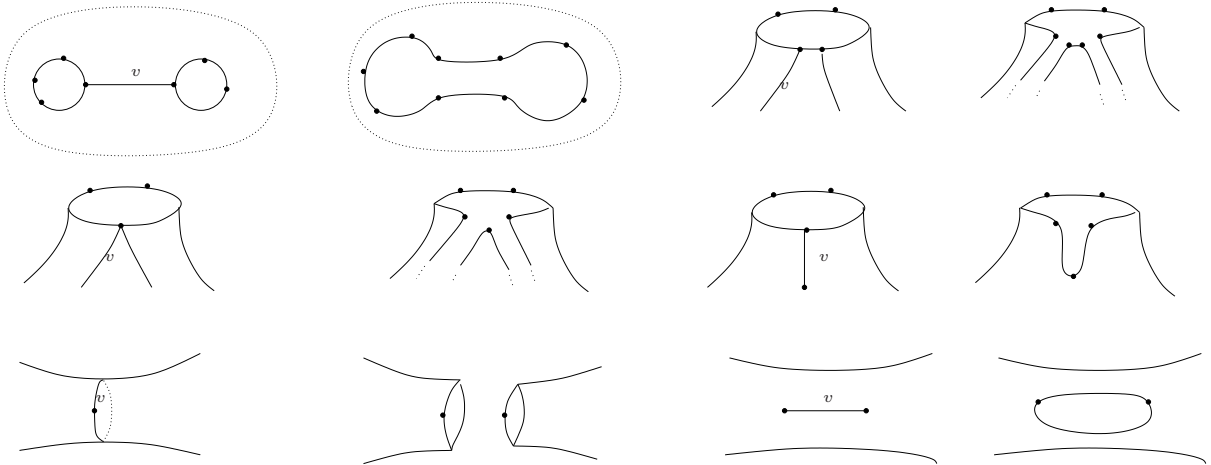}
\caption{Cutting $S$ along $v$}\label{arco-1}
\end{center}
\end{figure}


\begin{remark}[Join structure of links]
The following holds:
\begin{enumerate}
\item If $v$ is a separating arc in $(S,\p)$, then $Lk(v) \cong A_1 \star A_2$, where $A_1$ and $A_2$ are the arc complexes of the two connected components of $S \setminus v$. If none of the connected components is a 3-polygon, $A_1, A_2$ are non-empty and $Lk(v)$ has diameter 2.  
\item If $v$ is a non-separating arc in $(S,\p)$, then $Lk(v) \cong K$ , where $K$ is the arc complex of $S \setminus v$ considered as a ciliated surface.  
\end{enumerate}
\end{remark}

The following lemma characterizes the vertices of $A(S,\p)$ with low-dimensional link. 
\begin{lemma}\label{remark1bis}
The following holds:
\begin{enumerate}
\item Let $v, w$ be two vertices in $A(S, \p)$. If $Lk(v)$ and $Lk(w) $ coincide as subsets of $A(S, \p)$, then  $v = w$ or $Lk(v) = Lk(w) = \varnothing$.
\item If $v$ is a vertex such that $Lk(v)= \varnothing$ then $(S,\p)$ is a one-punctured disk with one marked point on the boundary or a $4$-polygon.  The converse also holds. 
\item If $v$ is a vertex such that $Lk(v)= \{ w \}$ (i.e. it consists of one single vertex) then $(S,\p)$ is a one-punctured disk with two marked point on its boundary.  The converse also holds. 
\item If  $v$ is a vertex such that $Lk(v) = \{ z, w \}$ (i.e. it consists of exactly two disjoint vertices) then $(S,\p)$ is  either a 5-polygon, or a one-punctured disk with  2 marked points on the boundary, or an annulus with one marked point on each boundary component.  
\end{enumerate}
\end{lemma}

\begin{proof}
(1): The cases where $\dim A(S, \p) \leq 1$ can be checked directly from Figure \ref{low-dim-1}. Assume $\dim A(S, \p) \geq 2$, it follows that $Lk(v) = Lk(w) $ has dimension at least 1. Let $\sigma$ be a maximal simplex in $Lk(v) = Lk(w) $. By construction $\tau_{v} = \langle v, \sigma \rangle$ and $\tau_w = \langle w, \sigma \rangle$ are two simplices of maximal dimension in $A(S, \p)$, and represent two ideal triangulations of $(S,\p)$ that differ by one arc. By Proposition \ref{connected}, we have $\iota(v, w)  = 1$. Using surgery, we can construct a loop disjoint by $v$ (in a tubular neighborhood of $v$) that intersects $w$. This loop is an element of $Lk(v)$ but not of $Lk(w)$,  and we have a contradiction.    

To prove (2), (3) and (4), we remark that $\dim Lk(v, A(S, \p)) = \dim A(S, \p) - 1$.  the results follow immediately from the classification of the arc complexes of dimension 0 and 1 given in Lemma \ref{remark1}.\end{proof}

\subsection{The invariance lemmas}\label{sec:inv_lemma}
Here we will prove that the following classes of  vertices in $A(S,\p)$ are invariant under automorphisms: 
\begin{itemize}
\item  the vertices that correspond to the arcs surrounding an inner puncture (we will call them \emph{drops}, see Section \ref{sec:drops}); 
\item  the vertices that correspond to the arcs parallel to a segment of the boundary (we will call them \emph{petals}, see Section \ref{sec:petals});
\item  the vertices that correspond to the loops surrounding a boundary component (we will call them \emph{loop-islands}, see Section \ref{sec:edge-bridges}); 
\item  the vertices that correspond to non-separating arcs (see Section \ref{sec:2leaves}).
\end{itemize} 
We will use the maximal dimension of a simplex containing only drops/petals/loop-islands in order to count the number of punctures/marked points/boundary components. 
When $\dim A(S,\p) \leq 1$ their invariance can be checked directly by Lemma \ref{remark1}. In the following we will always assume $\dim A(S,\p) \geq 2$. \
\subsubsection{Simplicial invariance of edge-drops}\label{sec:drops}
\begin{definition}
An arc $l$ on $(S, \p)$ is a \emph{drop} if it is a simple loop that bounds a once-punctured disk (see Figure \ref{arco2}). An edge $\langle l,v \rangle$ in $A(S, \p)$ is an \emph{edge-drop} if $l$ is a drop and $v$ joins the basepoint of $l$ to the inner puncture as in Figure \ref{arco2}. 
\end{definition}
\begin{figure}[htbp]
\begin{center}
\psfrag{v}{\small $v$}
\psfrag{l}{$l$}
\includegraphics[width=7cm]{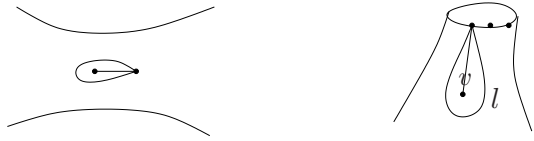}
\caption {Edge-drops  }\label{arco2}
\end{center}
\end{figure}

In the following we prove that the arc complex of a once-punctured disk $A(S_{0,1}^1,(1))$ behaves like a prime factor in the join product of two arc complexes. 
This will provide a simplicial characterization of drops. 
 
\begin{lemma}\label{cone}
The following holds:
\begin{enumerate}
\item $A(S,\p)$ is a cone if and only if $(S,\p)$ is a once-punctured disk with one marked point on the boundary, \emph{i.e.} if and only if $A(S,\p)$ consists of one single vertex. 
\item The join of two non-empty arc complexes is  a cone if and only if one of the two arc complexes consists of one single vertex.
\item  an arc $l$ is a drop if and only if its link in $A(S, \p)$ is isomorphic to a cone. 
\end{enumerate}
\end{lemma}

\begin{proof}
(1): If $A(S,\p)$ is a cone on a vertex $a$, then by definition $a$ is an arc on $(S,\p)$ disjoint by every other arc in $(S,\p)$. This can happen only if $(S,\p)$ has a unique essential arc, that is, if and only if $A(S,\p)$ consists of one single vertex.

(2): Let $A_1, A_2$ be arc complexes such that $A_1 \star A_2$ is a cone on a vertex $a$. If $a \in A_1$, then $a$ is disjoint by all the arcs in $A_1$. By (1), this is possible only if $A_1 = \langle a \rangle$. 

(3): Assume that $Lk(l)$ in $A(S, \p)$ is a cone. If $l$ is a non-separating arc, then $Lk(l)$ is isomorphic to the arc complex of $S \setminus l$. By (1), $S \setminus l$ would be isomorphic to a once-punctured disk with one point on its boundary, this would be in contradiction with the assumption that $l$ is nonseparating.  Assume that $l$ is a separating arc, and $Lk(l) = A_1 \star A_2$ where $A_1$ and  $A_2$ are the arc complexes of the connecting components of $S \setminus l$.  By (1) and (2) we deduce that one of the two connecting components of $S \setminus l$ is a once punctured disk with one point on its boundary. The converse of the statement is immediate.
\end{proof}

\begin{lemma}[Edge-drops invariance] \label{separating}
Let $\phi: A(S, \p) \to  A(S', \p')$ be an isomorphism.
The following holds: 
\begin{enumerate}
\item If $l$ is a drop, then $\phi(l)$ is a drop. 
\item If $\langle l,v \rangle$ is an edge-drop, then $\langle \phi(l), \phi(v) \rangle$ is an edge-drop. 
\end{enumerate}  
\end{lemma}

\begin{proof}
By Lemma \ref{cone} drops can be recognized by their link, and we have (1). Assertion (2) follows by (1) and the isomorphism between $Lk(l)$ and $Lk(\phi(l))$. 
\end{proof}

\begin{corollary}\label{punctures-s}
If $A(S_{g,b}^s,\p)$ is isomorphic to $A(S_{g',b'}^{s'},\p')$ then $s'=s$.
\end{corollary}

\begin{proof} 
The number of punctures on a surface is equal to the dimension of a simplex spanned by a maximal set of pairwise disjoint drops on it. By Lemma \ref{separating} drops are simplicial invariants, and so it is this simplex. 
\end{proof}

\subsubsection{Simplicial invariance of petals}\label{sec:petals}
\begin{definition}
Let $\mathscr B$ be a boundary component of $(S,\p)$ with $p \geq 2$ marked points on it. A $p$-\emph{leaf} on $(S,\p)$ is a simple loop based on $\mathscr B$ and parallel to $\mathscr B$. If $p \geq 3$ and $3\leq j\leq p$, a $j$-\emph{petal} on $(S,\p)$ is a separating arc parallel to $\mathscr B$ that bounds a $j$-polygon  (see Figure \ref{leaves_petals}).   
\begin{figure}[htbp]
\begin{center}
\includegraphics[width=6cm]{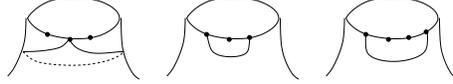}
\caption{A $3$-leaf, a $3$-petal and a $4$-petal }\label{leaves_petals}
\end{center}
\end{figure}
\end{definition}

We will prove that $j$-petals and $p$-leaves are invariant. We will first provide a simplicial characterization of $3$-leaves and $4$-petals. The key ingredient is that the arc complex of  the 4-polygon $A(S_{0,1}^0,(4))=\{v, w \}$  behaves like a prime factor for the join products of arc complexes. 
\begin{lemma} \label{remark2-NEW}
The following holds: 
\begin{enumerate}
\item Let $K=  \{ v , w \}$ be a simplicial complex that consists of two disjoint vertices. Let $H \neq \varnothing$ be a non-empty arc complex. If $K', H'$ are arc complexes such that $K  \star H$ is isomorphic to  $K' \star H'$, then $K'$ is isomorphic $K$  and $H'$ is isomorphic to $H$ (up to reordering). In particular, $K'$ and $H'$ are both non-empty. 
\item  An arc $z$ is a $4$-petal or a $3$-leaf on $(S,\p)$ if and only if $Lk(z, A(S, \p))= \{v, w\} \star H$, where $H \neq \varnothing$ is the arc complex of the surface obtained cutting $(S,\p)$ along $z$.
\end{enumerate}
\end{lemma}

\begin{proof}
(1)  Note that it cannot be that both $K', H' = \varnothing$. Denote by $v'$ and $w'$ the isomorphic images of $v$ and $w$ in $K' \star H'$. Note that we have $Lk(v, K \star H) = H = Lk(w, K\star H)$, in particular we have: 
\begin{align}
Lk( v, K \star H) &= Lk( w , K \star H)  \mbox{ as subsets of } K \star H \\
Lk( v', K' \star H') &= Lk( w' , K' \star H') \mbox{ as subsets of } K' \star H'. \label{eq:remark2}
\end{align}
Now assume by contradiction that $H' = \varnothing$ and $K' \neq \varnothing$. Now $K' \star H' = K'$ is an arc complex that contains two vertices $v'$ and $w'$ whose links coincide as subsets of $K'$. By Lemma \ref{remark1bis} both links are empty, hence $K'$ has dimension $0$. By Lemma \ref{remark1} $K'$ has exactly 2 vertices but $K\star H$ has at least 3 vertices, contradiction. We thus have $K' \neq \varnothing$ and $H' \neq \varnothing$. Now remark that in $K\star H$ there is no edge between $v$ and $w$, and by isomorphism the same holds for $v'$ and $w'$, so if $v'$ belongs to $K'$ then also $w'$ belongs to $K'$. 
We have:  
\begin{align}
 Lk(v', K' \star H') &= Lk(v', K') \star H' \\ 
 Lk(w', K' \star H') &= Lk(w', K') \star H'\\ 
  Lk(v', K') \star H' &=Lk(w', K') \star H' \mbox{ as subsets of } K' \star H'
\end{align} 
It follows that $Lk(v', K') = Lk(w', K')$, and  $K'$ is an arc complex containing two different vertices $v',w'$ whose links coincide as subset of $K'$. By Lemma \ref{remark1bis} then $Lk(v', K') = Lk(w', K') = \varnothing$ and $K' = \{ v' , w' \} \cong K$. It follows that $H' \cong H$.  

(2) Assume that $Lk(z, A(S, \p)) = \{ v, w \} \star H$, with $H \neq \varnothing$. If  $S \setminus z$ is a connected subsurface, then $Lk(z, A(S,\p)) \cong A(S\setminus z)$, in contradiction with (1). So $z$ is a separating arc and $Lk(z ,A(S,\p))$ is isomorphic to $A \star B$, where $A$ and $B$ are the non-empty arc complexes of the two connected components of $S \setminus z$. By (1) we have $A \cong \{ v, w\}$ and $B \cong H$. By Lemma \ref{remark1bis} one of the connected components of $S \setminus z$ is isomorphic to the 4-polygon $(S_{0,1}^0, (4))$, so $z$ is a 4-petal or a 3-leaf.    
\end{proof}

The following lemma can be adapted from Ivanov \cite{Iv2} (see also Irmak-McCarthy \cite{IMC}).
\begin{lemma}[Flip invariance]\label{int_1}
Let $\phi: A(S, \p) \to A(S', \p')$ be an isomorphism. 
Let $a, b \in A(S, \p)$ be vertices such that $i(a, b)=1$, then $i(\phi(a),\phi(b))=1$.
\end{lemma}

\begin{proof}
Since $\phi$ is an isomorphism, $\mathrm{ dim }A(S,\p) = \mathrm{ dim }A(S',\p')$ and $\phi$ sends maximal simplices into maximal simplices, that is, triangulations of $(S, \p)$) into triangulations of $(S', \p')$. Let $a$ and $b$ be arcs intersecting exactly once, we can extend $a$ to a triangulation $\tau_a$ such that the set of arcs $\tau_b:=(\tau_a \setminus \{a \}) \cup b$ is also a triangulation of $S$. Let  $\sigma$ be the simplex of $A(S, \p)$ defined as $\sigma = \tau_a \cap \tau_b = \tau_a \setminus a = \tau_b \setminus b$, it has codimension $1$. Now $\phi(\tau_a)$ and $\phi(\tau_b)$ are triangulations of $(S', \p')$, and $\phi(\sigma) = \phi (\tau_a)\cap \phi(\tau_b)= \phi(\tau_a) \setminus \phi(a) = \phi(\tau_b) \setminus \phi(b)$ has codimension $1$. By Proposition \ref{connected}  $\phi(\tau_a)$ and $\phi(\tau_b)$ differ by one elementary move, so $i(\phi(a), \phi(b))=1$.

\end{proof}

\begin{lemma}[3-petals invariance]\label{4}
Let $\phi: A(S, \p) \to A(S', \p')$ be an isomorphism. 
The following holds:
\begin{enumerate}
\item If $v$ is a $4$-petal, then $\phi(v)$ is a $4$-petal. 
\item If $v$ is a $3$-leaf, then $\phi(v)$ is a $3$-leaf.
\item If $v$ is a $3$-petal, then $\phi(v)$ is a $3$-petal. 
\item if $v$ is a $3$-petal based on a boundary component with $p\geq 3$ marked points, then $\phi(v)$ is a $3$-petal based on a boundary component with $p'=p$ marked points.
\end{enumerate}
\end{lemma}

\begin{proof}
1.
By Lemma \ref{remark2-NEW} the image under $\phi$ of a 3-leaf can be either a 3-leaf or a 4-petal. If $p_i \leq 3$ for all $i$'s then $(S, \p)$ does not contain any 4-petal, and we are done. Similarly, if $p_i \neq 3$ for all $i$'s then $(S,\p)$ does not contain any 3-leaf, and again we're done. 
Therefore, up to reorder the $p_i$'s, we will assume that $p_1 \geq 4$ and $p_2 =3$. We will now prove that $\phi$ cannot exchange 4-petals and 3-leaves.

Let $v$ be a 4-petal based on the first boundary component $\mathscr B_1$. Denote by $\rho$ a $\frac{2\pi}{p_1}$-rotation around $\mathscr B_1$, and remark that $\rho^k(v)$ is a 4-petal for all $k = 0, \ldots, p_1 -1$. The intersection pattern of all the 4-petals around $\mathscr B_1$ is the following: 
$$ i(\rho^k(v), \rho^h(v))= \begin{cases} 
1 & \mbox{, if } |k - h| = 1, 2 \mbox{ mod } p_1  \\ 
0 & \mbox{, otherwise } 
\end{cases} $$

By Lemma \ref{int_1} and simpliciality, the $\phi(\rho^k(v))$'s are $p_1$ distinct arcs with the same intersection pattern, and they are all $4$-petal or $3$-leaves. We deduce that they are necessarily parallel to the same boundary component $\mathscr B_i'$ that contains $\phi(v)$. If $\phi(v)$ is a $3$-leaf then $p_i'=3$ and $\mathscr B_i'$ supports at most three 3-leaves and no 4-petals, since $p_1 \geq 4$ we get to a contradiction. Thus $\phi(v)$ is a 4-petal, and so are all the $\phi(\rho^k(v))$'s. This argument also proves that $p_i' = p_1$.  

2. It follows immediately by (1) and  Lemma \ref{remark2-NEW}. 

3. 
Remark that if $v$ is a $3$-petal based on $\mathscr B_i$, there exists a $4$-petal (or a $3$-leaf, in the case $p_i=3$) $w$ based on $\mathscr B_i$ such that $\mathrm{Lk }(w, A(S,\p))= \{ v, \rho(v)    \} \star A$, where $A$ is the arc complex of the other connected component of $S \setminus w$. By the previous cases, $\phi(w)$ is also a $4$-petal (or a $3$-leaf when $p_i =3$), and $Lk(\phi(w), A(S', \p')) = \{ v', \rho(v') \} \star A'$, where $v'$ is a 3-petal in the complement of $S \setminus \phi(w)$ and $A'$ is the arc complex of the other connected component. By Lemma \ref{remark2-NEW} $\{ \phi(v), \phi(\rho(v)) \} = \{ v', \rho(v') \}$, and  $\phi(v)$ is a 3-petal.

4. The number $p_i$ of points on the $i$-th boundary component of $S$ is equal to the number of $3$-petals based on it.
\end{proof}
The following corollary can be proven by induction on $p$. 
\begin{corollary}[$p$-petals invariance]\label{inv:petals2}
Let $\phi: A(S, \p) \to  A(S', \p')$ be an isomorphism. If $v$ is a $p$-petal for $p\geq 3$, then $\phi(v)$ is also a $p$-petal. Similarly, if $v$ is a $p$-leaf for $p \geq 3$, then $\phi(v)$ is also a $p$-leaf.
\end{corollary}

\begin{corollary}\label{41}
If $A(S,\p)$ is isomorphic to $A(S',\p')$ then $(S, \p)$ and  $(S', \p')$ have the same number of boundary components  with $k$ marked points on it for every $k \geq 3$.\end{corollary}

\begin{proof}
Let $\phi$ be an isomorphism between $A(S,\p)$ and $A(S',\p')$. Denote by $\mathscr B_1, \ldots , \mathscr B_m$ the boundary components with exactly $k $ marked points on it. For each $i = 1, \ldots , m$,  let $\mathscr P_i$ be the set of all the 3-petals around $\mathscr B_i$. Note that if $i \neq j$ then $\mathscr P_i \cap \mathscr P_j = \varnothing$ and $\phi(\mathscr P_i) \cap \phi(\mathscr P_j) = \varnothing$. By Lemma \ref{4} each element in $\phi(\mathscr P_i)$ is a 3-petal parallel to a boundary component with $p' = k$ marked points on it. Thus $(S', \p')$ has at least $m$ boundary components with exactly $k$ marked points. The converse inequality uses the same argument for $\phi^{-1}$.    
\end{proof}

\begin{corollary}\label{inv:petals1}
If $\phi: A(S,\p) \to A(S',\p')$ is an isomorphism, the following holds: 
\begin{enumerate}
\item If $a$ is a loop based on a boundary component $\mathscr B$ with $p \geq 3$ marked points  then $\phi(a)$ is a loop based on a boundary component $\mathscr B'$ with $p' = p$ marked points. 
\item If $a$ is an arc with exactly one endpoint on $\mathscr B$ based on with $p\geq 3$ marked points then $\phi(a)$ is an arc with exactly one endpoint on $\mathscr B'$ with $p' = p$ marked points.
\item If $a$ is an arc with exactly two endpoints on $\mathscr B$ based on with $p \geq 3$ marked points then $\phi(a)$ is an arc with exactly two endpoints on $\mathscr B'$ with $p' = p$ marked points.
\end{enumerate}
\end{corollary}

\begin{proof}
Use $3$-petals to mark the endpoints of $a$ and use Lemma \ref{4} - (3).  
\end{proof}

\subsubsection{Simplicial invariance of loop-islands}\label{sec:edge-bridges}
Assume that $(S_{g,b}^s, \p)$ is a surface such that $b+s \geq 2$ and at least one boundary component has exactly one marked point on it. 
\begin{definition}
A loop $l$ in $(S,\p)$ is a  \emph{loop-island} if $l$ is a separating arc such that $S \setminus l$ contains an annulus with exactly one marked point on each boundary component.
An edge $\langle l, w \rangle$ of $A(S, \p)$ is an \emph{edge-bridge} if $l$ and $w$ are as in Figure \ref{archi2}, that is, $l$ is a  loop-island and $w$ is a non-separating arc in the annular component of $S \setminus l$. We will call $w$ \emph{arc-bridge}.
\begin{figure}[htbp]
\begin{center}
\psfrag{v}{\tiny $w$}
\psfrag{l}{ \tiny $l$}
\includegraphics[width=4cm]{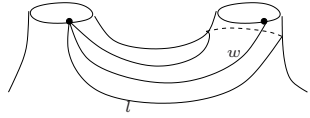}
\caption{An edge-bridge}\label{archi2}
\end{center}
\end{figure}
\end{definition}

We will prove that the arc complex of  the annulus with exactly one marked point on each boundary component (i.e. $A(S_{0,1}^0,(1,1)) = \mathbb R$) behaves like a prime factor in the join product of arc complexes. This will provide a simplicial characterization of edge-bridges.

\begin{lemma}\label{remark2-bis-NEW}
The following holds:
\begin{enumerate}
\item Set $K = \mathbb R$, \emph{i.e.} the simplicial complex isomorphic to $\mathbb R$ and set $H = A(S, \p)$. 
If $K'$ and $H'$ are arc complexes such that $K \star H$ is isomorphic $K' \star H'$, then $K$ is isomorphic $K'$ and $H$ is isomorphic to $H'$ (up to reordering). In particular, $K'$ and  $H'$ are both non-empty.  
\item  A vertex $l$ is a loop-island if and only if $Lk(l, A(S,\p)) \cong \mathbb R \star H$, where $H\neq \varnothing$ is the arc complex of the other connected component of $S \setminus l$. 
\end{enumerate}
\end{lemma}

\begin{proof}
(1) If $H = \varnothing$, then $K' \star H' \cong \mathbb R$, so $\dim K', \dim H' \leq 1$. By Lemma \ref{remark1} $K' \cong \mathbb R$ and $H' \cong \varnothing$ up to renaming and we're done. We now assume $H = \varnothing$. Up to renaming, we can also assume that  infinitely many vertices of $K'$ are isomorphic images of vertices in $\mathbb R$. By Lemma \ref{remark1} $K'$ has dimension at least $1$. Let $z \in \mathbb R$ be a vertex such that its isomorphic image $z'$ is in $K'$.  
We can assume that $z'$ does not correspond to a $4$-petal or a $3$-leaf on the underlying surface of $K'$ (there are only finitely many petals or leaves in $K'$). 
We have:
\begin{align}\label{eq:remark3}  
\{ v, w\} \star H =Lk(z , \mathbb R \star H) \cong Lk(z', K' \star H') = Lk(z', K') \star H' .
\end{align} 
The vertex $z'$ corresponds to a separating or non-separating arc on the underlying surface of $K'$, and  we have two cases on $Lk(z', K')$: 
\begin{description}
\item{(a)} $Lk(z', K') = A$, where $A$ is a non-empty arc complex ($z'$ is nonseparating);
\item{(b)} $Lk(z',K') = A \star B $, where $A$ and $B$ are non-empty arc complexes ($z'$ is separating).
\end{description}
Let us first deal with the case (a). By Lemma \ref{remark2-NEW} we have two more cases: 
\begin{description}
\item{(a.1)} $Lk(z', K') \cong  \{ v, w \}$ (and $H' \cong H$);  
\item{(a.2)} $H' \cong \{v, w\}$ (and $Lk(z', K') \cong H$).   
\end{description}
In case (a.1) $K'$ has dimension 1. By $\mathbb R$ is the unique 1-dimensional arc complex with infinitely many vertices, so $K' \cong \mathbb R$, and 
we have proved the assertion. In case (a.2) we have $H' = \{v, w\} $ and $\mathbb R \star H \cong K' \star \{ v,w \}$.  By Lemma \ref{remark2-NEW} we have $H = \{ v,w \}$. It follows $K' = \mathbb R $ and $H' \cong H$, and we have proved the assertion.  \\
Let us now deal with case (b). Here we have $ \{v, w \} \star H = A \star B \star H' .$ 
By the same argument in the proof of Lemma \ref{remark2-NEW} exactly one of the following holds: 
\begin{description}
\item{(b.1)}  $H' = \{v, w\}$;  
\item{(b.2)}  $A= \{v, w\}$; 
\item{(b.3)}  $B = \{v,w\}$.
\end{description}
In the case (b.1)  we have $\mathbb R \star H \cong K' \star \{v, w\} $, hence $K' \cong \mathbb R$ by Lemma \ref{remark2-NEW}. In the cases (b.2) and (b.3) we conclude $z'$ is a 3-leaf or a 4-petal on $K'$ by Lemma \ref{remark2-NEW}. This contradicts our assumption on $z'$.  \\ 

(2)  We will now prove that if $l$ is a vertex such that $Lk(l, A(S,\p)) \cong \mathbb R \star H$, then $l$ is a loop-island. The converse is immediate. Note that $l$ is not a non-separating arc nor a 2-leaf nor a 3-petal, otherwise its link would be some non-empty arc complex $K'$, and $K' \cong \mathbb R \star H$ is in contradiction with (1). Thus $l$ is a separating arc and $Lk(l, A(S,\p)) \cong K' \star H'$, where $K'$ and $H'$ are the non-empty arc complexes of its connected components. Again by (1), we have $K' \cong \mathbb R$ and $H' \cong H$. By Lemma \ref{remark1} we deduce that the surface whose arc complex is $K'$ is an annulus $(S_{0,1}^0, (1,1))$ and we conclude that $l$ is a loop-island.
\end{proof}

\begin{lemma}[Edge-bridges invariance]\label{edge-bridge}
Let $\phi:A(S,\p) \to A(S,\p)$ be an isomorphism. If  $\langle l, w \rangle$ is an edge-bridge, then $\langle \phi(l), \phi(w) \rangle$ is an edge-bridge. Moreover, $\phi(l)$ is a loop-island,  and $\phi(w)$ is an arc-bridge. 
\end{lemma}

\begin{proof}
If follows from Lemma \ref{remark2-bis-NEW} and the fact that  the link of $\phi(l)$ is isomorphic to the link of $l$. 
\end{proof}

\begin{corollary}\label{edge-bridge2}
If $A(S,\p)$ is isomorphic to $A(S',\p')$ then $(S, \p)$ and $(S', \p')$ have the same number of boundary components with one marked point on it.
\end{corollary}

\begin{proof}
Let $\epsilon_1, \ldots, \epsilon_k$ be pairwise disjoint edge-bridges on $(S,\p)$ with $k=k(S,\p')$ the largest possible.  If $s \neq 0$ or $\p \neq (1, \ldots, 1)$ then $k$ is 
exactly the number of boundary components $b$ of $(S,\p)$ with exactly one marked point.  If $s=0$ and $\p = (1, \ldots ,1)$,  then $k = b-1$. 
We will prove that 
\begin{enumerate}
\item $k(S,\p) = k(S', \p')$ whenever $A(S,\p)$ is isomorphic to $A(S', \p')$;
\item if $s=0$ and $\p=(1, ..., 1)$ then $s' = 0 $ and $\p' =(1, \ldots 1)$. 
\end{enumerate}
By the above remark, (1) and (2) suffice to prove the corollary. 
Let us prove 1. Denote by $\phi$ an isomorphism between $A(S,\p)$ and $A(S',\p')$. By simpliciality and Lemma \ref{edge-bridge} $\phi(\epsilon_1), \ldots ,\phi(\epsilon_k)$ are 
also pairwise disjoint edge-bridges on $S'$ and $k(S',\p') \geq k(S,\p)$.  Using the same argument on $\phi^{-1}$, we have that $k(S,\p) \geq k(S',\p')$, hence $k(S,\p) = k(S',\p')$. \\ 
Let us prove 2. We say that an edge-bridge $\langle l, v \rangle$ is reversible if there exists a loop-island $l' \neq l$ such that $\langle v, l' \rangle$ is also an edge-bridge. 
Equivalently, an edge-bridge is reversible if and only if $v$ is an arc-bridge between two boundary components both having exactly one marked point on it.  It is immediate to see that we have $s=0$ and $\p =(1, 
\dots, 1)$ if and only if every edge-bridge $\epsilon_i$ in a family of pairwise disjoint edge-bridges $\epsilon_1, \ldots, \epsilon_k$ with $k$ maximal is reversible. 
By Lemma \ref{edge-bridge} these condition are invariant under isomorphism, so $\p = (1, \ldots, 1) $ if and only if $\p'= (1, \ldots, 1)$.  
\end{proof}

\subsubsection{Simplicial invariance of nonseparating arcs and Theorem 1.1}\label{sec:2leaves}
Here we will prove that nonseparating arcs are invariant and we will complete the proof of Theorem 1.1. Let us first extend the results of Corollary \ref{41} to 2-leaves. 
\begin{lemma}[2-leaves invariance]\label{2-leaves}
Let $\phi: A(S,\p) \to A(S', \p')$ be an isomorphism. If $l$ is a 2-leaf, then $\phi(l)$ is a 2-leaf. 
\end{lemma}

\begin{proof}
Let us first introduce the following notation: if $k$ is a positive integer,  $\beta^k(\p)$ is the number of boundary components of $(S, \p)$ with exactly $k$ marked points, and   $\beta^{\geq k}(\p)$ the number of boundary components with  at least $k$ marked points. By  Corollary \ref{41} and Corollary \ref{edge-bridge2} we have:
\begin{align}
\mathbb{\beta}^1(\p) &=\mathbb{\beta}^1(\p') \\ 
\mathbb{\beta}^k(\p) &=\mathbb{\beta}^k(\p') \mbox{ for all } k \geq 3.
\end{align}
When $l$ be a 2-leaf, then $S\setminus l$ is a connected surface of type $(S, \mathbb T)$. We have:  
\begin{align}\label{beta}
\beta^{1}(\mathbb T) &= \beta^{1}(\mathbb P) + 1 \\ 
\beta^{2}(\mathbb T) &= \beta^{2}(\mathbb P) - 1   \\ 
\beta^{k}(\mathbb T) &= \beta^{k}(\mathbb P) \mbox{ for all } k \geq 3
\end{align} 
We will prove that $\phi(l)$ is a 2-leaf using the invariants $\beta^i$. By the invariance lemmas proved so far, we have two main cases: 
\begin{description}
\item{(a.)}  $\phi(l)$ is non-separating;
\item{(b.)}  $\phi(l)$ is separating. 
\end{description}
Let us first prove that the case (a) can never occur. In this case $\phi(l)$ is non-separating, so $S' \setminus \phi(l)$ is a connected surface with marked points on the boundary. Denote by $\mathbb{T'}$ its vector of marked points on the boundary. The map $\phi$ restricts to an isomorphism between the arc complexes of $S \setminus l$ and $S \setminus \phi(l)$, so 
$$ \beta^{k}(\mathbb T) = \beta^{k}(\mathbb T')  ~\mathrm{ for }~ k =1~ \mathrm{ and }~ k \geq 3 .$$  
If $\phi(l)$ is non-separating only one of the following occurs:  
\begin{description}
\item{(a.1)} $\phi(l)$ is an arc joining two different points of the same boundary component of $S'$;
\item{(a.2)} $\phi(l)$ is an arc joining two different boundary components of $S'$; 
\item{(a.3)} $\phi(l)$ is a loop based on a boundary component;
\item{(a.4)} $\phi(l)$ is a loop based on a puncture; 
\item{(a.5)} $\phi(l)$ is an arc joining a boundary component and a puncture. 
\end{description}
We will now see that each of the cases above leads to a contradiction. (a.1) When we cut $S'$ along $\phi(l)$ we destroy exactly one boundary component with at least two 
marked points and we create exactly two new boundary components, both with at least 2 marked points, so $\beta^1(\mathbb{T}')=  \beta^1(\mathbb{P}') \neq \beta^1(\mathbb{T}) 
$, and we get a contradiction. (a.2) By Lemma \ref{edge-bridge} $\phi(l)$ cannot be a arc-bridge, so each boundary component that contains the endpoints of $\phi(l)$ has at least 
two marked points. When we cut along $\phi(l)$, we merge these two boundary components into a new one with at least 6 marked points, again we have a contradiction: $
\beta^{1}(\mathbb T' ) = \beta^{1}(\p') \neq  \beta^{1}(\mathbb T)$. (a.3) Assume that $\phi(l)$ is based on a boundary component with $p'$ marked points on it. When we cut along 
$\phi(l)$, we destroy one boundary component and create two new boundary components one with exactly marked points on it and the other with $p' +1$ marked points on it. If 
$p'=1$ then $\beta^1(\mathbb{T}') = \beta^{1}(\p') \neq \beta^1(\mathbb{T})$. If $p' \geq 2$ then $\beta^{\geq 3}(\mathbb{T}') \geq \beta^{\geq 3}(\p') +1 \neq \beta^1(\mathbb{T})$. 
In both cases we have a contradiction. 
(a.4) $S' \setminus \phi(l)$ has two new boundary components each with exactly one marked point on it, and $\beta^1(\mathbb{T}')= 2+ \beta^1(\mathbb{P}') \neq \beta^1(\mathbb{T})$, contradiction.  (a.5) It contradicts Lemma \ref{separating}. \\

Let us now study the case (b). First write $S' \setminus \phi(l) = S_1' \cup S_2'$.  If $A(S_1') = \varnothing$, then $S_1'$ is a 3-polygon and $\phi(l)$ is a 3-petal or a 2-leaf. By 
Lemma \ref{4}, $\phi(l)$ cannot be a 3-petal and we are done. Assume now assume that both $A(S_1') $ and $A(S_2')$ are not empty. 
By Lemmas \ref{4} and \ref{separating} $\phi(l)$ is not a petal, not a drop, therefore $S_1',S_2'$ are both different from a polygon or once-punctured polygon. Let $\pi$ be a 
maximal collection of pairwise disjoint petals and leaves on $(S,\p)$ such that $l \in \pi$. Its complement $S\setminus \pi$ is a connected surface with at most one point on each 
boundary component, and $Lk(\pi, A(S,\p)) \cong A(S, (1, 
\ldots, 1))$. By Lemma \ref{4} the simplex $\phi(\pi)$ contains petals, leaves or separating arcs (the only possible images of the $2$-
leaves of $\pi$). Set $K = Lk(\phi(\pi) \cap A(S_1'), A(S_1')) $ and $H = Lk(\phi(\pi) \cap A(S_2'), A(S_2'))$. We have: 
\begin{equation}\label{eq:diam}
Lk(\pi, A(S,\p)) \cong A(S, (1, \ldots, 1))\cong Lk(\phi(\pi), A(S', \p')) \cong K  \star H .
\end{equation}
By Remark \ref{dimension} $A(S, (1, \ldots, 1))$ has infinite diameter.  If both $K, H \neq \varnothing$, then $K \star H$ has diameter 2, and we have a contradiction.  
If $K = \varnothing$, then $\phi(\pi) \cap A(S_1')$ provides a triangulation of $S_1'$ where 
each arc is separating, so $S_1'$ is a polygon, and again we have a contradiction. The case $H = \varnothing$ is analogue.
\end{proof}

\begin{corollary}\label{2-leaves-2}
If $A(S,\p)$ is isomorphic to $A(S',\p')$ then $(S,\p)$ and $(S', \p')$ have the same number of boundary components with $2$ marked points. 
\end{corollary}

\begin{proof}
The number of boundary components of $(S,\p)$ with $2$ marked points  is equal to the maximum number of pairwise disjoint $2$-leaves that one can put on a surface. By Lemma \ref{2-leaves} it is simplicial invariant.  
\end{proof}

\begin{thma}
If  $A(S_{g,b}^s, \p)$ isomorphic to $A(S_{g',b'}^{s'}, \p')$, then $s=s'$, $b=b'$, $g=g'$ and $p_i=p_i'$ for all $i$ (up to reordering).
\end{thma}

\begin{proof}
The proof in the case $\mathrm{dim } A(S,\p) \leq 1$ follows from Lemma \ref{remark1}. In the other cases, the proof goes as follows. 
The equality $s=s'$ was proved in Corollary \ref{punctures-s}. The equalities $b=b'$ and $p_i=p_i'$ for all $i$ (up to reordering) follows from Corollary \ref{41} , \ref{edge-bridge2}, \ref{2-leaves-2}. Finally $g=g'$ follows from $\mathrm{dim} A(S_{g,b}^s, \p) = \mathrm{dim } A(S_{g',b'}^{s'}, \p')$. 
\end{proof}

The following invariance lemma for non-separating arcs will be useful later.
\begin{lemma}[Nonseparating arcs invariance]\label{nonseparating}
If $\phi: A(S,\p) \to A(S', \p')$ is an isomorphism and $l$ is a non-separating arc of $(S,\p)$, then $\phi(l)$ is also a non-separating arc. 
\end{lemma}

\begin{proof}
Let us first prove thet we can reduce to the case  $(S, \p) \cong (S, (1, \ldots, 1))$. 
Let $\pi$ be a maximal collection of pairwise disjoint petals and leaves of $(S,\p)$ that are also disjoint from $l$. Its complement $S \setminus \pi$ is a union of triangles and a connected surface of type $(S, (1, \ldots 1))$. Since $\mathrm{Lk}(\pi) \cong A(S, (1, \ldots, 1))$,  $l$ can be regarded as a non-separating arc on $A(S, (1, \ldots, 1))$. By isomorphism $\mathrm{Lk}(\phi(\pi)) \cong \mathrm{Lk}(\pi) \cong A(S, (1, \ldots, 1))$ with $\phi(l) \in \mathrm{Lk}(\phi(\pi)) \cong A(S, (1, \ldots, 1))$, so we can just prove the lemma for $(S, \p) \cong (S, (1, \ldots, 1))$. 

Now work in the case $(S, \p) \cong (S, (1, \ldots, 1))$. We have two main cases: 
\begin{description}
\item{(a)} $l$ has two different endpoints; 
\item{(b)} $l$ is a loop.  
\end{description}
In the case (a), $\phi(l)$ is an arc of the same type by the lemma \ref{separating} and \ref{edge-bridge}. \\
In the case (b), we will assume that $\phi(l)$ is separating and argue by contradiction. Assume that $\phi(l)$ separates $S$ in two connected components. By Lemmas \ref{inv:petals2} and \ref{2-leaves} both components have non-empty arc complexes (say, $H$ and $K$), so $Lk(\phi(l)) \cong H \star K$ has diameter 2.  By isomorphism, $Lk(l)$ has also diameter 2. We will now see that this is impossible. We have two cases: 
\begin{itemize}
\item $l$ is a loop based on a puncture;
\item $l$ is a loop based on a boundary point. 
\end{itemize} 
If $l$ is a loop based on a puncture, then $S \setminus l$ is a connected surfaces with one marked point on each boundary component, so $\mathrm{Lk}(l) \cong A(S \setminus l)$  has infinite diameter by Lemma \ref{infinite diameter}, a contradiction. If $l$ is a loop based on a boundary component, let $l'$ be a non-separating loop parallel to the union of $l$ and the boundary component that contains the endpoints of $l$. The complement $S \setminus \{ l, l' \}$ is the union of a triangle and a surface $S_{\langle l, l' \rangle}$ with one marked point on each boundary component, and $\mathrm{Lk}(\langle l, l' \rangle) \cong A(S_{\langle l, l' \rangle})$ has infinite diameter. As in the previous case, we have: 
$$A(S_{\langle l, l' \rangle}) \cong \mathrm{Lk}(\langle l, l' \rangle) \cong \mathrm{Lk}(\langle \phi(l), \phi(l') \rangle) \cong H' \star K' $$ 
where $H', K'$ are the non-empty arc complexes of the connected components of $S \setminus \{ \phi(l), \phi(l')\}$. We note that $A(S_{\langle l, l' \rangle})$ has infinite diameter and $H' \star K'$ have diameter 2, a contradiction. 
\end{proof}

\section{Proof of Theorem 1.2}\label{section3}
In this section we will prove the following theorem. 
\begin{thmb}
Every automorphism of the arc complex $A(S,\p)$ is induced by a mapping class of $(S,\p)$. Moreover, if the dimension of $A(S,\p)$ is greater than 1, the automorphism group of $A(S,\p)$ is isomorphic to the mapping class group of $(S,\p)$. 
\end{thmb}
Let $\phi: A(S,\p) \to A(S,\p)$ be an automorphism. We will construct a triangulation $\tau$ and a homeomorphism $\Psi:(S,\p) \to (S,\p)$ such that $\Psi$ and $\phi$ agree on every arc of $\tau$. By the following lemma \ref{Ivanov} $\phi$ is induced by $\Psi$ (see also \cite{IMC} and \cite{Iv2}). By this and Proposition \ref{one-to-one} we have Theorem 1.2.
\begin{lemma}\label{Ivanov}
Let $\tau$ be a triangulation of $(S,\p)$. If $\Psi: (S,\p) \to (S,\p)$ is an homeomorphism such that $\phi(t) = \Psi(t)$ for every arc $t \in \tau$, then $\phi$ is the automorphism of $A(S,\p)$ induced by the homeomorphism $\Psi$. 
\end{lemma}

\begin{proof}
We will prove that $\phi(v) = \Psi(v)$ for every vertex $v \in A(S,\p)$. Fix a vertex $v \not \in \tau$ and extend it to a triangulation $\tau'$. If $\tau $ and $\tau'$ differ by one flip, then necessarily that flip involves $v$ and $v$ is disjoint by all the arcs in $\tau$ except one that intersects exactly once. By Lemma \ref{int_1} and simpliciality the same holds for $\phi(v)$ and $\phi(\tau)$, so $\phi(\tau')$ differs by $\phi(\tau)=\Psi(\tau)$ by one flip that involves $\phi(v)$.  By construction $\phi(\tau') \cap \phi(\tau) = \Psi(\tau') \cap \Psi(\tau) = \Psi(\tau') \cap \phi(\tau)$ is a simplex of codimension 1, therefore $\phi(v) = \Psi(v)$. In general, $\tau $ and $\tau'$ differ by a finite number of flips. The same argument proves that $\phi$ and $\Psi$ coincide after each flip, by induction $\phi(\tau') = \Psi(\tau')$ and $\phi(v) = \Psi(v)$.
\end{proof}
\subsubsection{Construction of $\tau$}
We will construct a triangulation $\tau$  with the property that for each triangle $\Delta$ in $S \setminus \tau$, the arcs in the simplex $\phi(\Delta)$ also bound a triangle in $S \setminus \phi(\tau)$. It will be useful to recall that there are three types of triangles in $(S,\p)$ and these are coded by three types of simplices of $A(S,\p)$:
\begin{itemize}
\item a single vertex $\langle x \rangle$ (triangles bounded by a $3$-petal or a $2$-leaf);  
\item an edge $\langle x, y \rangle$ (triangles with one edge on $\partial S$ or self-folded like in edge-drops);
\item a 2-simplex $\langle x, y, z \rangle$ (triangles with 3 essential edges). 
\end{itemize}
We will need the following lemma: 
\begin{lemma}\label{prelemma}
Let $\phi:A(S,\p) \to A(S,\p)$ be an isomorphism.  If $v,w$ are arcs with a common endpoint on the boundary of $S$, so are $\phi(v)$, $\phi(w)$. 
\end{lemma}

\begin{proof}
If this boundary component have at least two marked points, then there exists a $3$-petal or a $2$-leaf $z$ intersecting them both, i.e. $i(z,v) \neq 
0$ and $i(z, w) \neq 0$. By simpliciality and invariance lemmas \ref{41} the same holds for $\phi(z)$, $\phi(v)$, $\phi(w)$, so $\phi(v)$ and $\phi(w)$ have a common endpoint on the boundary. If the boundary component has only one marked point, the same argument holds when $z$ is an appropriate loop-island. \end{proof}
We now construct $\tau$ in four steps. 
\paragraph{Step 1} Enclose all the marked points on the boundary of $S$. \\   
Skip this step if every boundary component of $(S,\p)$ has exactly one marked point on it. Otherwise, triangulate a neighborhood of each boundary component $\mathscr B_i$ with $p_i \geq 2$ marked points  in order to enclose them all  as in Figure \ref{Step 1}. More precisely, for each $\mathscr B_i$ construct a collection $\pi_i$ of pairwise disjoint arcs as follows: 
\begin{itemize}
\item choose a point $P_i$ among its marked points, pick the $p_i$-leaf $v_{p_i}^i$ based in $P_i$;
\item choose a point $Q_i$ adjacent to $P_i$ and pick the 3-petal  $v_3^i$ that has an endpoint on $P_i$  and encloses $Q_i$;
\item triangulate the $(p_i+1)$-polygon bounded by  $v_{p_i}^i$ with petals with one endpoint on $P_i$ as follows: pick the $j$-petal $v_j^i$ which has an endpoint in $P_i$ and encloses $v_3^i$ for every $j=3 , \ldots , p_i-1$.
\end{itemize}

\begin{figure}
\begin{center}
\psfrag{l}{\small $v_{p_i}^i$}
\psfrag{v}{\small $v_{3}^i$}
\psfrag{P}{\small  $P_i$}
\psfrag{Q}{\small  $Q_i$}
\includegraphics[width=5cm]{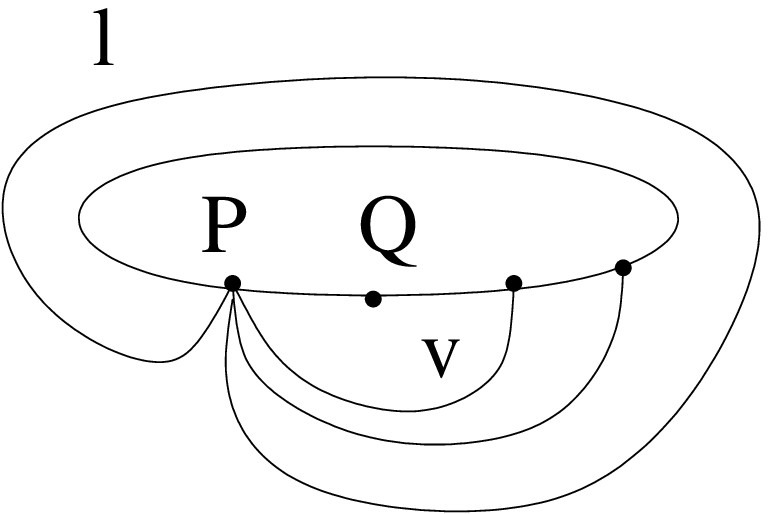}
\caption{Step 1 : }\label{Step 1}
\end{center}
\end{figure}

The arcs in $\pi_i$ span a simplex in $A(S,\p)$ that contains exactly one $p_i$-leaf and one $j$-petal for each $j = 3, \ldots, p_i-1$. Its complement $S \setminus \pi_i$ contains exactly $p_i-1$ triangles: one for the vertex $\langle v_3^i \rangle$, the others for the pairs $\langle v_j^i, v_{j+1}^i \rangle$ for $j = 3, \ldots , p_i-1$.   By  the invariance lemmas \ref{inv:petals2} or \ref{2-leaves}, we have: 
\begin{lemma}\label{configuration 1}
The configuration of arcs in $\pi_i$ is $\phi$-invariant: 
\begin{itemize}
 \item $\phi$ maps arcs in $\pi_i$ to arcs in $\phi(\pi_i)$ of the same topological type; 
 \item if $\Delta$ is a triangle in $S \setminus \pi_i$ then $\phi(\Delta)$ is also a triangle.
 \end{itemize} 
\end{lemma}

Denote by $\pi$ the union of all the $\pi_i$'s: $\pi$ also spans a simplex in $A(S,\p)$. Its complement $S \setminus \pi$ is a union of triangles and a surface $S_\pi$ with exactly one marked point on each boundary component. The map $\phi$ induces isomorphisms on each row of the following diagram (here $1_b$ denotes the vector $(1, \ldots, 1)$): 
$$\xymatrix{
Lk(\pi) \ar[d]^{\cong} \ar[r]^\phi  & Lk(\phi(\pi)) \ar[d]^\cong \\
A(S_\pi) \ar[r] \ar[d]^\cong & A(S_{\phi(\pi)})  \ar[d]^\cong \\ 
A(S, 1_b) \ar[r] & A(S, 1_b)
}$$
\paragraph{Step 2} Enclose all the  boundary components in a neighborhood of $\mathscr B_1$. \\ 
Skip this step if $(S,\p)$ has exactly one boundary component. Otherwise, proceed as follows. Let $P_1, P_2$ be the marked points on $\mathscr B_1$ and $\mathscr B_2$ as in Step 1. Triangulate a neighborhood of $\mathscr B_1$ to enclose $\mathscr B_2$ as in Figure \ref{Step 2}. More precisely, construct a collection $\beta_1$ of pairwise disjoint arcs as follows: 
\begin{itemize}
\item choose a loop $l$ based in $P_1$ that surrounds $\mathscr B_2$, and choose two arcs $ v,w $ from $P_1$ to $P_2$ so that $\langle l, w \rangle$ and $\langle l, v \rangle$ are edge-bridges based in $S_\pi$;
\item  choose a loop $z$  based in $P_1$ and parallel to $\mathscr B_1 \cup l$. 
\end{itemize}
 
\begin{figure}
\begin{center}
\psfrag{P}{\tiny $P_1$}
\psfrag{v}{\tiny $v$}
\psfrag{u}{\tiny $u'$}
\psfrag{u'}{\tiny $u$}
\psfrag{w}{\tiny $w$}
\psfrag{l}{\tiny $l$}
\psfrag{z}{\tiny $z$}
\includegraphics[width=8cm]{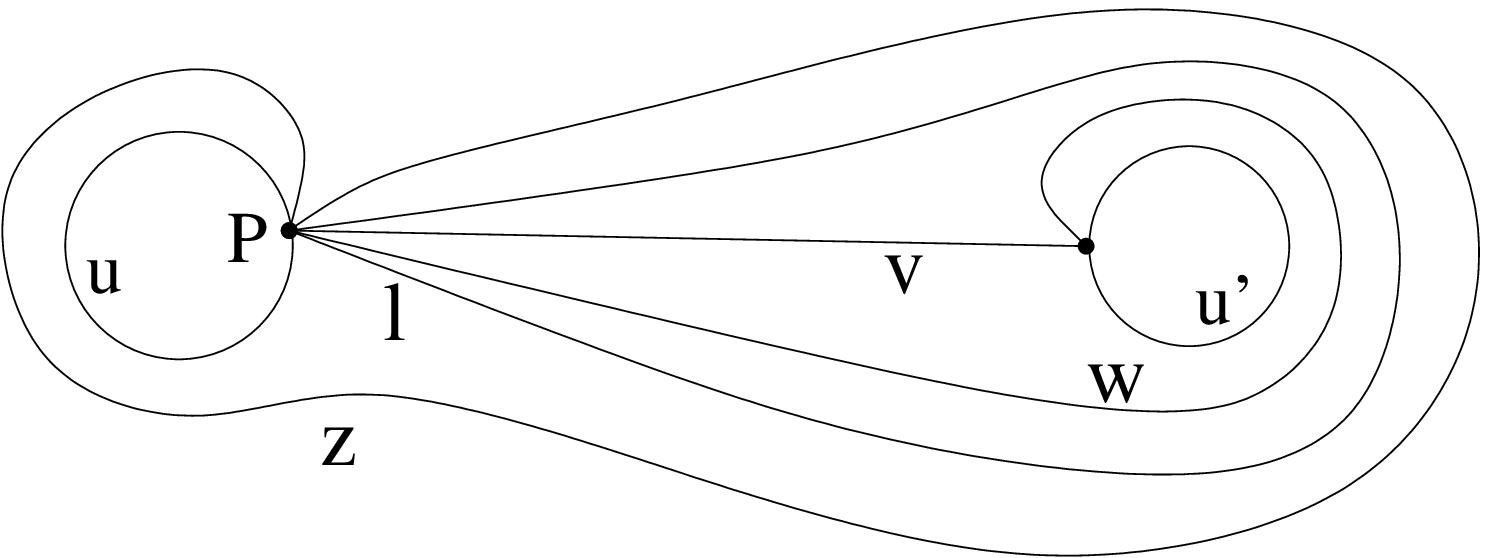}
\caption{Step 2}\label{Step 2}
\end{center}
\end{figure}

The collection $\beta_1$ spans a simplex in $A(S, \p)$. All the arcs in $\beta_1$ have an endpoint in $P_1$ and $S \setminus \{ \pi,  \beta_1 \}$ contains 3 new triangles corresponding to the triplets $\langle v, w, u \rangle$, $\langle l, v , w\rangle$, $\langle z, l , u' \rangle$ (here $u$ can be either $\mathscr B_2$ or the outer loop in $\pi_2$ as in Step 1, and $u'$ can be either $\mathscr B_1$ or the outer loop in $\pi_1$ as in Step 1). 
\begin{lemma}\label{configuration 3}
The configuration of the arcs in $\mathscr \beta_1$ is $\phi$-invariant: 
\begin{itemize}
 \item $\phi$ maps every arc in $\beta_1$ to an arc in $\phi(\beta_1)$ of the same topological type; 
 \item if $\Delta$ is a triangle in $S \setminus \beta_1$ then $\phi(\Delta)$ is also a triangle.
 \end{itemize} 
\end{lemma}

\begin{proof}
 Let us work in $Lk(\pi) \cong A(S, 1_b)$. Under this isomorphism, $l$ is a loop-island, $v$ and $w$ are arc-bridges, $\langle l, v \rangle$ and $\langle l, w \rangle$ are edge-bridges. By the invariance lemma 
\ref{edge-bridge}, the arc $\phi(l)$ is a loop-island, the arcs $\phi(v), \phi(w)$ are arc-bridges, the pairs $\langle \phi(l), \phi(v) \rangle$, $\langle \phi(l) ,\phi(w) \rangle$ are edge-bridges in $A(S, 1_b) \cong Lk(\phi(\pi))$. So $\langle  \phi(v), \phi(w), \phi(l) \rangle$, $\langle  \phi(v), \phi(w) \rangle$ are triangles in $S \setminus \phi(\pi)$ and $\langle  \phi(v), \phi(w), \phi(u) \rangle$ is a triangle in $S$. \\
Let us now prove that $z$ and $\phi(z)$ have the same topological type and $\langle \phi(u'), \phi(l), \phi(z) \rangle$ is also a triangle. The complement $ S \setminus \{ \pi, l, v, w \}$ contains a surface $S_{\langle \pi, l, v, w \rangle}$ that has $b-1$ boundary components, one with two marked points and the others with one marked point. By Theorem 1.1 the same holds for $S \setminus \{ \phi(\pi), \phi(l), \phi(w), \phi(v) \}$. We have isomorphisms : 
$$\xymatrix{
Lk(\langle \pi, l, v, w \rangle) \ar[d]^{\cong} \ar[r]  & Lk(\langle \phi(\pi), \phi(l), \phi(v), \phi(w) \rangle) \ar[d]^\cong \\
A(S_{ \langle \pi, l, v, w \rangle}) \ar[r]_{\cong} \ar[d]^\cong & A(S_{\langle \phi(\pi), \phi(l), \phi(v), \phi(w) \rangle})  \ar[d]^\cong \\ 
A(S, (2,1_{b-2})) \ar[r]_{\cong} & A(S, (2,1_{b-2}))
}$$

The arc $z$ can be seen as a $2$-leaf in $A(S, (2, 1_{b-2}))$. By Lemma \ref{2-leaves} $\phi(z)$ is also a 2-leaf on $A(S, (2, 1_{b-2}))$ and it bounds a triangle with two edges on the boundary, and these edges correspond to $\phi(l)$ and $\phi(u')$. Hence $\langle \phi(l), \phi(z), \phi(u') \rangle$ is also a triangle. 
By Lemma \ref{prelemma} all the arcs in $\phi(\beta_1)$ have a common endpoint $P_1'$. 
\end{proof} 

Now iterate this same construction on $S \setminus \{ \pi, \beta_1 \}$:  at each step construct a collection of arcs $\beta_{i}$ with the same properties of $\beta_1$. The union of $\beta_1, \ldots, \beta_b$ spans a simplex $\beta$ in $A(S,\p)$. Note that $S \setminus \{ \pi, \beta \}$ is a union of triangles and a surface $S_{ \langle \pi, \beta \rangle}$ with one boundary component with one marked point and $s$ punctures.  We have: 
$$\xymatrix{
Lk(\pi \cup \beta) \ar[d]^{\cong} \ar[r]^\phi  & Lk(\phi(\pi \cup \beta)) \ar[d]^\cong \\
A(S_{\langle \pi, \beta \rangle}) \ar[d]^\cong \ar[r]^\cong &A(S_{\langle \phi(\pi), \phi(\beta) \rangle}) \ar[d]^\cong \\ 
A(S_g^s, (1)) \ar[r]^\cong  & A(S_g^s, (1))
}$$
\paragraph{Step 3} Enclose each puncture in a neighborhood of $\mathscr B_1$ as in Figure \ref{Step 3}. \\
Skip this step when $(S,\p)$ has no punctures. Otherwise, triangulate a neighborhood of $\mathscr B_1 $ to enclose each puncture as in Figure \ref{Step 3}. Choose a puncture and construct a collection of arcs $\sigma_1 = \{ l, v, z \}$ as follows: 
\begin{itemize}
\item $\langle l, v \rangle$ is an edge-drop based in $P_1$; 
\item  $z$ is a loop based in $P_1$ and parallel to $\mathscr B_1 \cup l$.
\end{itemize}

\begin{figure}
\begin{center}
\psfrag{P}{\small $P_1$}
\psfrag{v}{\small $v$}
\psfrag{w}{\small $w$}
\psfrag{u}{\small $u$}
\psfrag{l}{\small $l$}
\psfrag{z}{\small $z$}
\includegraphics[width=6cm]{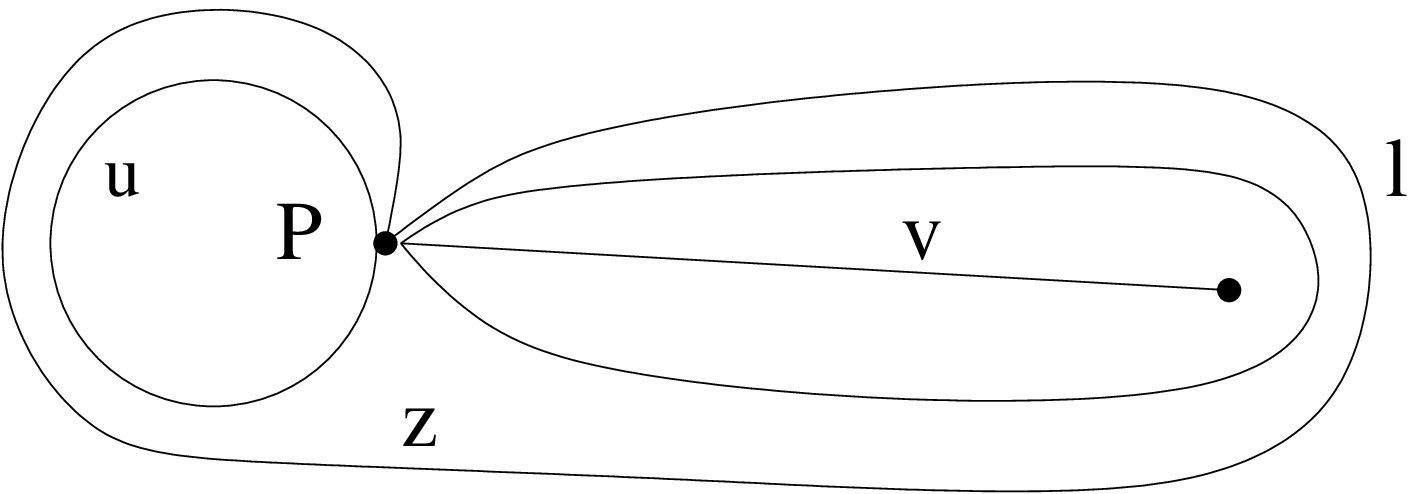}
\caption{Step 3}\label{Step 3}
\end{center}
\end{figure}

The arcs in $\mathscr \sigma_1$ have one endpoint in $P_1$ and span a simplex in $A(S,\p)$. The surface $S \setminus \{ \pi, \beta , \sigma_1 \}$ has two new triangles: one for the triplet $\langle l, z, u \rangle$ (here $u$ can be  $\mathscr B_1$, the outer loop in $\pi$ or the outer loop in $\beta$  ) and one for the pair $\langle v, l \rangle$.

\begin{lemma}
The configuration of arcs in $\sigma_1$ is $\phi$-invariant: 
\begin{itemize}
 \item $\phi$ maps every arc in $\sigma_1$ to an arc in $\phi(\sigma_1)$ of the same topological type; 
 \item if $\Delta$ is a triangle in $S \setminus \sigma_1$ then $\phi(\Delta)$ is a triangle.
 \end{itemize} 
\end{lemma}

\begin{proof}
By Lemma \ref{separating}, $\phi(l)$ and $\phi(v)$ have the same topological type of $l$ and $v$, respectively, and $\langle \phi(l), \phi( v) \rangle$ is also a triangle. 
Let us now prove that $z$ and $\phi(z)$ have the same topological type. The complement $S \setminus \{ \pi,  \beta, l, v\} $ contains a surface $S_{\langle \pi,  \beta, l, v\rangle}$ that has $s-1$ punctures and one boundary component with 2 marked points on it. By Theorem 1.1, the same holds for $S \setminus \{ \phi(\pi), \phi(\beta), \phi(l), \phi(v) \}$. We have the following isomorphisms: 
$$\xymatrix{
Lk(\langle \pi, \beta, l, v \rangle ) \ar[d]^{\cong} \ar[r]  & Lk(\langle \phi(\pi),  \phi(\beta), \phi(l), \phi(v) \rangle) \ar[d]^\cong \\
A(S_{\langle \pi, \beta, l, v\rangle}) \ar[r]_\cong \ar[d]^\cong & A(S_{ \langle \phi(\pi), \phi(\beta), \phi(l), \phi(v)\rangle}) \ar[d]^\cong \\
A(S_g^{s-1}, (2)) \ar[r]_\cong  &  A(S_g^{s-1}, (2))
}$$
The arc $z$ can be seen as a $2$-leaf in $A(S_g^{s-1}, (2))$, by Lemma \ref{2-leaves} $\phi(z)$ is also a $2$-leaf, and $\phi(z)$ bounds a triangle with two edges on the boundary. So  $\langle \phi(z), \phi(l), \phi(u) \rangle$ also bound a triangle. 
By Lemma \ref{prelemma} $\phi(l)$, $\phi(v)$ and $\phi(\pi)$ have a common endpoint $P_1'$.
\end{proof}

Now iterate this construction on $S \setminus \{ \pi, \beta, \sigma_1 \}$ and construct a collection of arcs $\sigma_i$ with the same properties of $\sigma_1$ for each puncture. The union of the $\sigma_1, \ldots, \sigma_s$ spans a simplex $\sigma$ in $A(S,\p)$. The complement $S \setminus \{ \pi, \beta, \sigma \}$ is a union of triangles and a surface $S_{\langle \pi, \beta, \sigma \rangle}$ having one boundary component with one marked point on it and no puncture in its interior. We have: 
$$\xymatrix{
Lk(\pi \cup \beta \cup \sigma ) \ar[d]^{\cong} \ar[r]^\phi  & Lk(\phi(\pi \cup \beta \cup \sigma)) \ar[d]^\cong \\
A(S _{\langle \pi, \beta , \sigma \rangle}) \ar[r]^\cong \ar[d]^\cong & A(S_{\langle \phi(\pi), \phi(\beta), \phi(\sigma)\rangle}) \ar[d]^\cong \\
A(S_g, (1)) \ar[r]^\cong  & A(S_g, (1))
}$$
\paragraph{Step 4} Enclose the non-separating loops. \\ 
Skip this step if $(S,\p)$ has genus 0. Otherwise, work in the complement $ S \setminus \{ \pi, \beta, \sigma \}$. Choose a non-separating loop $a$ based in $P_1$. Triangulate a neighborhood of $\mathscr B_1$  to enclose $a$ as in Figure \ref{Step 4}. More precisely, construct a collection of pairwise disjoint arcs $\gamma_1= \{ a, b, v, w, z \}$ as follows: 
\begin{itemize}
\item $b$ is an arc  parallel to $\mathscr B_1 \cup a$;   
\item $l$ is a loop that surrounds the boundary component relative to $a$ in $S \setminus \{ \pi, \beta, \sigma, a \}$; 
\item $v, w, z$ are arcs as in Step 3 that enclose the boundary component relative to $a$ in $S \setminus \{ \pi, \beta, \sigma, a \}$.
\end{itemize}

\begin{figure}
\begin{center}
\psfrag{P}{\small $P_1$}
\psfrag{a}{\small $a$}
\psfrag{b}{\small $b$}
\psfrag{l}{\small $l$}
\psfrag{u}{\small $u$}
\includegraphics[width=2cm]{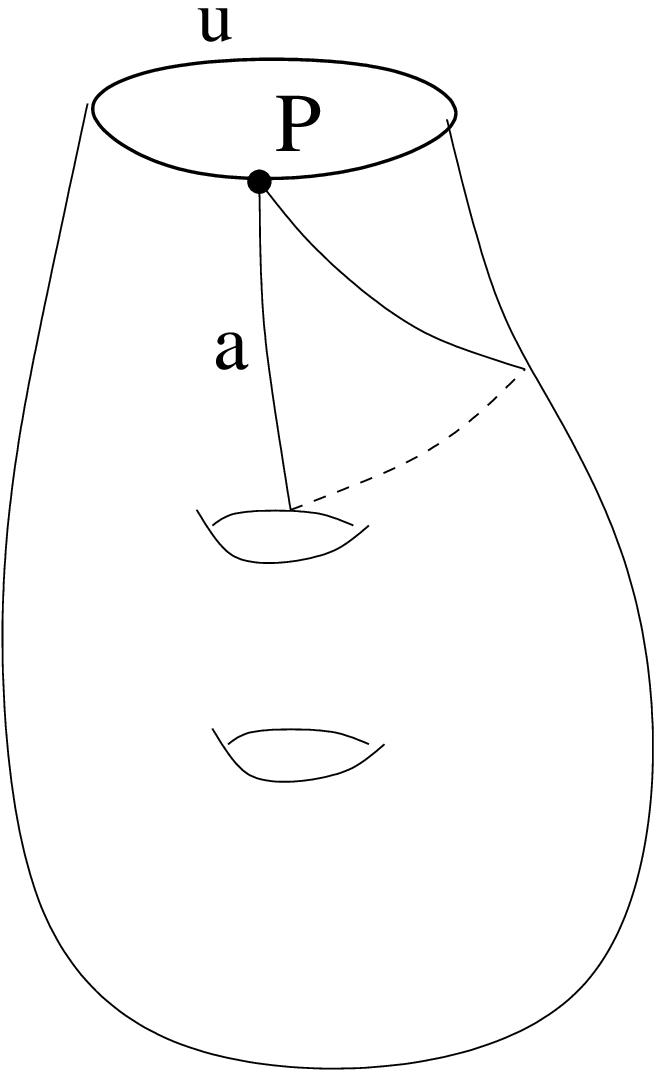}
\includegraphics[width=6cm]{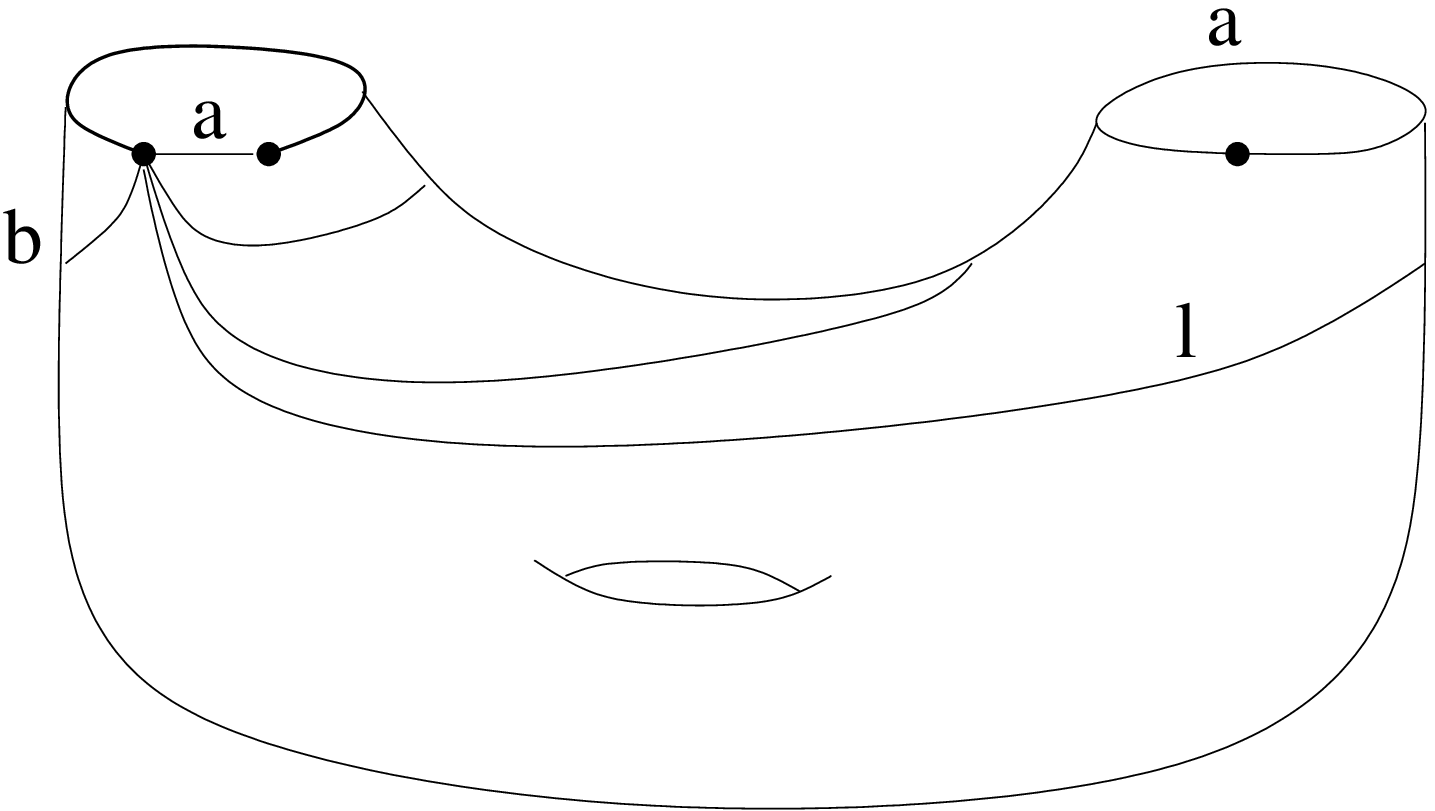}
\caption{Step 4}\label{Step 4}
\end{center}
\end{figure}

The arcs in $\gamma_1$ span a simplex in $A(S,\p)$ and have at least one endpoint on $P_1$. $S \setminus \{ \pi , \beta, \sigma, \gamma_1 \}$ contains four new triangles coming from $\gamma_1$: one for each triplet $\langle l, v, w \rangle$, $\langle b, l, z  \rangle$, and $\langle a, b , u \rangle$ (here $u$ is either $\mathscr B_1$ or the outer loop in one of the above Steps) and the pair $\langle v , w \rangle$.   
We have the following:  
\begin{lemma}\label{configuration 4} 
The configuration of the arcs in $\mathscr \gamma_1$ is $\phi$-invariant: 
\begin{itemize}
 \item $\phi$ maps arcs in $\gamma_1$ to arcs in $\phi(\gamma_1)$ of the same topological type; 
 \item if  $\Delta$ is a triangle in $S \setminus \gamma_1$ then $\phi(\Delta)$ is also a triangle.
 \end{itemize} 
\end{lemma}

\begin{proof}
The arc $a,\phi(a)$ are both non-separating by Lemma \ref{nonseparating}. We have isomorphisms: 
$$\xymatrix{
Lk(\langle \pi, \beta, \sigma, a \rangle ) \ar[d]^{\cong} \ar[r]  & Lk(\langle \phi(\pi),  \phi(\beta), \phi(\sigma), \phi(a) \rangle) \ar[d]^\cong \\
A(S _{\langle \pi, \beta, \sigma, a\rangle}) \ar[r]_\cong \ar[d]^\cong & A(S_{\langle \phi(\pi), \phi(\beta), \phi(\sigma), \phi(a)\rangle}) \ar[d]^\cong \\
A(S_{g-1}, (2,1)) \ar[r]_\cong  &  A(S_{g-1}, (2,1))
}$$
The arc $b$ can be thought of as a $2$-leaf in $A(S_{g-1}, (2,1))$. By Lemma \ref{2-leaves} $\phi(b)$ is also a $2$-leaf in $A(S_{g-1}, (2,1))$, it bounds a triangle with the two segments on its boundary, so $\langle \phi(a), \phi(b), \phi(u) \rangle$ also bound a triangle. The rest of the proof follows from Lemma \ref{configuration 3} up to passing to $Lk(\langle \pi, \beta, \sigma, a, b \rangle) \cong A(S_{g-1}, (1,1))$. 

\end{proof}

Now iterate this construction on $S \setminus \{ \pi , \beta , \sigma, \gamma_1 \}$ : at each step choose a new non-separating arc and you get a collection of arcs $\gamma_i$ with the same properties of $\gamma_1$. The union of $\gamma_1, \ldots, \gamma_g$ obtained this way spans a simplex $\gamma$ in $A(S,\p)$. 
The union $ \pi \cup \beta \cup \sigma \cup \gamma$ spans a simplex $\tau$ in $A(S,\p)$ which is a triangulation of $(S,\p)$. 

\subsubsection{Construction of $\Psi$}
By Lemmas \ref{configuration 1}, \ref{configuration 3}, \ref{configuration 4} if $\Delta$ is a triangle in $S \setminus \tau$, then $\phi(\Delta)$ is a triangle of the same type in $S \setminus \phi(\tau)$. We want construct an homeomorphism $\Psi$ that "covers" $\phi$. \\ For every $\Delta$ in $S \setminus \tau$, we can find homeomorphisms from $\Delta$ to $\phi(\Delta)$ as follows: 
\begin{itemize}
\item for every triangle $\Delta_i = \langle x, y, z \rangle$ there is a unique homeomorphism $ \Psi_i : (\Delta_i, x, y, z) \to (\phi(\Delta_i), \phi(x), \phi(y), \phi(z))$ well-defined up to isotopies and its orientation type is fixed; 
\item for every triangle $\Delta_i = \langle x, y \rangle$ with one edge on the boundary, there is a unique homeomorphism $\Psi_i : (\Delta_i, x, y) \to (\phi(\Delta_i), \phi(x), \phi(y))$ well-defined up to isotopies and its orientation type is fixed; 
\item for every triangle  $\Delta_i = \langle x, y \rangle$  corresponding to a drop, there exist two homeomorphisms $\Psi_i, \bar{\Psi_i}: (\Delta_i, x, y) \to (\phi(\Delta_i), \phi(x), \phi(y))$ well-defined up to isotopies and they have opposite orientation types;
\item for every triangle  $\Delta_i = \langle x \rangle$  corresponding to a $3$-petal or a $2$-leaf, there exist two homeomorphisms $\Psi_i, \bar{\Psi_i}: (\Delta_i, x) \to (\phi(\Delta_i), \phi(x))$ well-defined up to isotopies and they have opposite orientation types.
\end{itemize}
We will now see that the maps $\Psi_i$ associated to any two adjacent triangles are compatible on the common edge. In the lemma below we see that $\phi$ preserves the configuration of the arcs in a quadrilateral formed by any two adjacent triangles.  

 \begin{lemma}\label{lem:B}
Let $\Delta_1, \Delta_2$ be two adjacent embedded triangles with at most one edge on the boundary, and let $f \in \Delta_1 \cap \Delta_2$ be a common edge.  Let $f^\star$ be the edge obtained 
performing a flip on $f$ and $\Delta_1^\star$, $\Delta_2^\star$ be the two new triangles bounded by $f^\star$. Then $\phi(\Delta_1^\star)$ and $\phi(\Delta_2^\star)$ are both triangles and  $\phi(f^\star)$ is their common edge.
\end{lemma}
\begin{proof}
Let $Q$ be the quadrilateral $\Delta_1 \cup \Delta_2$, we will see that $\phi$ preserves the configuration of the arcs in $Q$. By Lemmas \ref{configuration 1}, \ref{configuration 3}, \ref{configuration 4} $\phi(Q)$ is a quadrilateral and $\phi(f)$ is a diagonal. By Lemma \ref{int_1} $\phi(f^\star)$ is obtained flipping $\phi(f)$. 
It is  immediate that if one between $\phi(\Delta_1^\star)$ or $\phi(\Delta_2^\star)$ is a triangle, then the other is also a triangle. Here we proceed by cases, and we study the five possible types of quadrilaterals that we can find in $\tau$. \\
The case (1) is the one in Figure \ref{lem:B}, where the arc $a$ can be either a piece of boundary or a petal, $b,d$ are segments of the boundary. The arc $f^\star$ is a $3$-petal with $\Delta_1^\star = \langle f^\star \rangle$ and $\Delta_2^\star = \langle a, f^\star, c \rangle$. By Lemma \ref{inv:petals2}, $\phi(f^\star)$ is also a $3$-petal, so $\phi(\Delta_1^\star)$ is also a triangle. \\ 
\begin{figure}[ht]
\begin{center}
\psfrag{d}{\small $d$}
\psfrag{a}{\small $a$}
\psfrag{b}{\small $b$}
\psfrag{f}{\small $f$}
\psfrag{g}{\small $f^\star$}
\psfrag{c}{\small $c$}
\includegraphics[width=10cm]{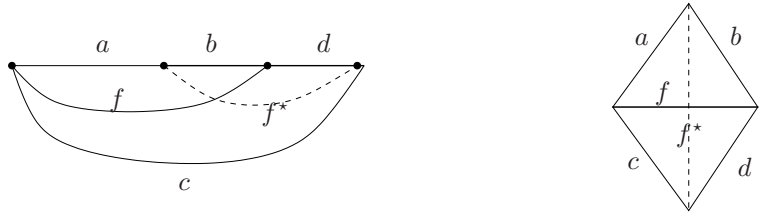}
\caption{Configuration (1)}\label{lem:B}
\end{center}
\end{figure} \\
The case (2) is the one in Figure \ref{lemB:2}, where the arcs $a,b$ can be either segments of the boundary or arcs in $\tau$ and $\Delta_1^\star = \langle a, z, f^\star \rangle$, $\Delta_2^\star = \langle b, l, f^\star \rangle$.  
Assume that $a$ is an arc of $\tau$ (the case where $a$ is a boundary component and $b$ is an arc of $\tau$ is analogue). We will prove that $\phi(\Delta_2^\star)$ is a triangle. By simpliciality $\phi(Q)$ is also a quadrilateral and $\phi(f)$ and $\phi(f^\star)$ are its two diagonals. If $\phi(\Delta_2)$ is not a triangle, then $\langle \phi(a), \phi(l), \phi(f^\star) \rangle$ is a triangle, so $\phi(f^\star)$ is a 3-petal in the connected component of $S \setminus \{ \phi(a), \phi(l) \}$ that contains $\phi(z)$ and $\phi(a)$. By the invariance lemma of 3-petals, the same holds for $f^\star$ in $S \setminus \{a, l \}$, and $f^\star$ should be parallel to $a \cup l$, contradiction.
\begin{figure}[htbp]
\begin{center}
\psfrag{z}{\small $z$}
\psfrag{l}{\small $l$}
\psfrag{a}{\small $a$}
\psfrag{b}{\small $b$}
\psfrag{f}{\small $f$}
\psfrag{g}{\small $f^\star$}
\psfrag{c}{\small $c$}
\includegraphics[width=10cm]{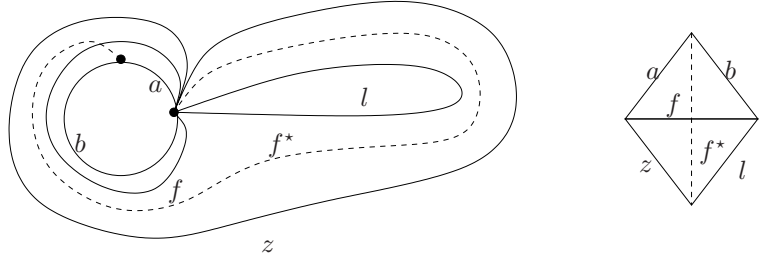}
\caption{Configuration (2)}\label{lemB:2}
\end{center}
\end{figure} \\
The case (3) is the one in Figure \ref{lemB:3}, where the arc $l$ is either a boundary component or an arc of $\tau$, and $\Delta_1^\star= \langle z, f^\star, a \rangle$, $\Delta_2^\star = \langle f^\star, l, a\rangle $. The assertion is immediate. 
\begin{figure}[h!]
\begin{center}
\psfrag{z}{\small $z$}
\psfrag{l}{\small $l$}
\psfrag{a}{\small $a$}
\psfrag{b}{\small $b$}
\psfrag{f}{\small $f$}
\psfrag{g}{\small $f^\star$}
\psfrag{c}{\small $c$}
\includegraphics[width=10cm]{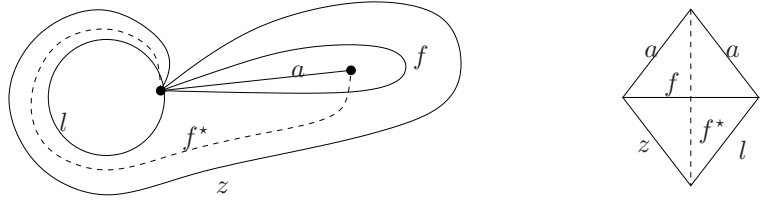}
\caption{Configuration (3)}\label{lemB:3}
\end{center}
\end{figure} \\
The case (4) is the one in Figure \ref{lemB:4}, where the arc $a$ is either a boundary component or an arc of $\tau$. Here $\Delta_1^\star = \langle a, b , f^\star  \rangle$ and $\Delta_2^\star = \langle b, f^\star, z \rangle$. If $\phi(\Delta_1^\star)$ is not a triangle, then $\langle \phi(b) , \phi(a), \phi(f^\star) \rangle$ is a triangle. The arcs $\phi(a)$ and $\phi(z)$ share an endpoint, so $a$ and $z$ by Lemma \ref{prelemma} and we have a contradiction.  
\begin{figure}[h!]
\begin{center}
\psfrag{z}{\small $z$}
\psfrag{l}{\small $l$}
\psfrag{a}{\small $a$}
\psfrag{b}{\small $b$}
\psfrag{f}{\small $f$}
\psfrag{g}{\small $f^\star$}
\psfrag{c}{\small $c$}
\includegraphics[width=10cm]{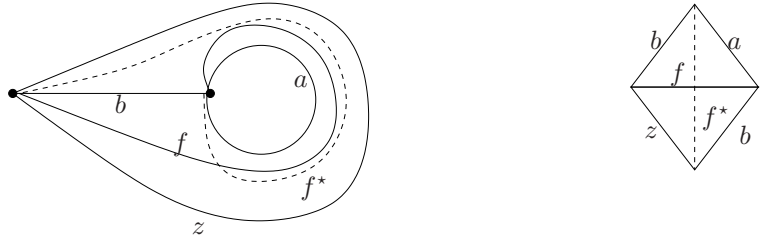}
\caption{Configuration (4)}\label{lemB:4}
\end{center}
\end{figure} \\
The case (5) is  the one in Figure \ref{lemB:5}, where $l$ is a boundary component or an arc of $\tau$. Here $\Delta_1^\star = \langle b , f^\star, z \rangle$, $\Delta_2^
\star = \langle a , f^\star, l \rangle$.  As in (2), if $\phi(\Delta_1)$ is not a triangle then $\langle \phi(a), \phi(z), \phi(f^\star) \rangle$ is a triangle.  $\phi(f^\star)$ is a 3-petal in the connected component of $S \setminus \{ \phi(a), \phi(z) \}$ that contains $\phi(b)$ and $\phi(l)$. By Lemma \ref{4} the same holds for $f^\star$ in $S \setminus \{a, z \}$, so $f^\star$ is parallel to $a \cup z$, contradiction.
\begin{figure}[h!]
\begin{center}
\psfrag{z}{\small $z$}
\psfrag{l}{\small $l$}
\psfrag{a}{\small $a$}
\psfrag{b}{\small $b$}
\psfrag{f}{\small $f$}
\psfrag{g}{\small $f^\star$}
\psfrag{c}{\small $c$}
\includegraphics[width=9.5cm]{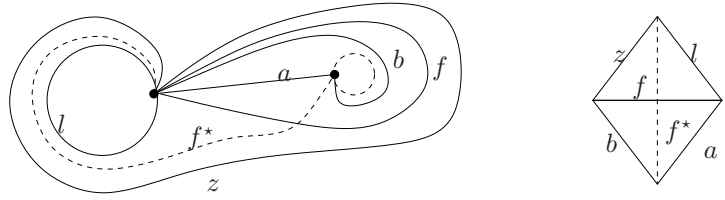}
\caption{Configuration (5)}\label{lemB:5}
\end{center}
\end{figure} 
\end{proof}

By Lemma \ref{lem:B}, we can choose the $\Psi_i$'s so that if two triangles $\Delta_i$ and $\Delta_j$ have a common edge then their associated homeomorphisms $\Psi_i$ and $\Psi_j$ agree on that edge,  and the maps can be glued together. The map $\Psi: (S, \p) \to (S,\p)$ obtained after all these glueings will be a well-defined homeomorphism (up to isotopies) which maps $\tau$ to $\phi(\tau)$. By Lemma \ref{Ivanov} $\Psi$ induces $\phi$, and we are done.

\addcontentsline{toc}{section}{References}
\bibliographystyle{amsplain}
\bibliography{references-new}
\vspace{4cm}
\textit{Address:} \\ 
Department of Mathematics \\
Indiana University, Bloomington IN, USA \\  
\url{vdisarlo@indiana.edu}
\end{document}